\documentclass[11pt, leqno]{article}
\usepackage{amsmath,amsfonts,amssymb,wasysym,mathrsfs}%,hyperref}
\usepackage[latin1]{inputenc}
\usepackage{shapepar}
\usepackage{graphicx}
\usepackage{hyperref}
\usepackage{ifpdf} %for .eps and .jpg
\usepackage{color}
\usepackage{mathrsfs}
\newcommand{\R}{{\mathbb R}}
\newcommand{\tr}{\operatorname{tr}}

\newcommand{\ren}{{\mathbb R}^N}
\newcommand{\N}{{\mathbb N}}

\newcommand{\be}[1]{\begin{equation}\label{#1}}
\newcommand{\ee}{\end{equation}}
\renewcommand{\(}{\left(}
\renewcommand{\)}{\right)}

\newcommand{\prf}{\par\smallskip\noindent{\sl Proof. \/}}
\newcommand{\finprf}{\unskip\null\hfill$\;\square$\vskip 0.3cm}

\newenvironment{proof}{\prf}{\finprf}
\newtheorem{theorem}{Theorem}[section]
\newtheorem{lemma}{Lemma}[section]
\newtheorem{corollary}[theorem]{Corollary}
\newtheorem{proposition}[theorem]{Proposition}
\newtheorem{remark}[theorem]{Remark}
\newtheorem{definition}{Definition}[section]

\numberwithin{equation}{section}

\newcommand{\Dom}{\operatorname{Dom}}
\def\qed{\,\unskip\kern 6pt \penalty 500
\raise -2pt\hbox{\vrule \vbox to8pt{\hrule width 6pt
\vfill\hrule}\vrule}\par}
\definecolor{darkblue}{rgb}{0.05, .05, .65}
\definecolor{darkgreen}{rgb}{0.1, .65, .1}
\definecolor{darkred}{rgb}{0.8,0,0}
\begin{document}
\title{\textbf{ Symmetrization for fractional Neumann problems}\\[7mm]}

\author{\Large Bruno Volzone\footnote{Dipartimento per le Tecnologie, Facolt\`a di Ingegneria, Università degli Studi di
Napoli ``Parthenope'', 80143  Italia. \    E-mail: {\tt bruno.volzone@uniparthenope.it}} }

\date{} %%  this cancels date in article format

\maketitle

\begin{abstract}
In this paper we complement the program concerning the application of symmetrization methods to nonlocal PDEs by providing new estimates, in the sense of mass concentration comparison, for solutions to linear fractional elliptic and parabolic PDEs with Neumann boundary conditions. These results are achieved by employing suitable symmetrization arguments to the Stinga-Torrea local extension problems, corresponding to the fractional boundary value problems considered. Sharp estimates are obtained first for elliptic equations and a certain number of consequences in terms of regularity estimates is then established. Finally, a parabolic symmetrization result is covered as an application of the elliptic concentration estimates in the implicit time discretization scheme.
\end{abstract}

\

\setcounter{page}{1}
%%%%%%%%%%%%%%%%%%%%%%%%%%%%%%%%%%%%%%%%%%%%%%%%%%%%%%%%%%%%%%%%%
\section{Introduction}\label{sec.intro}
Following the study initiated in the work \cite{dBVol} and continued in \cite{VazVol1}, \cite{VazVol2}, \cite{VazVolSire},  the spirit of this note is to
provide
a further insight on applications of classical symmetrization techniques to PDEs involving fractional Laplacian operators. In particular, we will focus on
these
methods with the aim of deriving optimal estimates, in the sense of mass concentration comparisons and their consequences, for solutions of nonlocal
elliptic
PDEs with \emph{Neumann} boundary conditions, of the type
\begin{equation} \label{eq.0}
\left\{
\begin{array}
[c]{lll}%
\left(  -\Delta\right)^{\sigma}u+cu=f\left(  x\right)   &  & in\text{ }%
\Omega,\\[10pt]
\dfrac{\partial u}{{\partial\nu}}=0 &  & on\text{ }\partial\Omega,
\end{array}
\right. %
\end{equation}
for all the exponents $\sigma\in(0,1)$.
Problem \eqref{eq.0} is posed in a smooth domain $\Omega$ of $\R^{N}$ ($N\geq2$), $\nu$ is the outward unit normal vector to $\partial \Omega$, $c$ is a
nonnegative constant and the source term $f=f(x)$ is assumed to belong (for instance) to $L^{p}(\Omega)$ for suitable $p>1$. When $c=0$, we will require the
natural compatibility condition
\begin{equation}
\int_{\Omega}f dx=0.\label{compatib}
\end{equation}
Using the results we shall achieve in the elliptic framework, we will determine also a comparison result for solutions to \emph{Cauchy-Neumann} linear
parabolic
problems of the form
\begin{equation}\label{linearparabintr}
\left\{
\begin{array}
[c]{lll}%
u_{t}+(-\Delta)^{\sigma}u =f &  & in\text{
}\Omega\times(0,T)\\[15pt]
\dfrac{\partial u}{\partial\nu}=0 &  & on\text{ }\partial \Omega\times[0,T],\\[15pt]
u(x,0)=u_{0}(x)   & &
in\text{ }\Omega,
\end{array}
\right.  %
\end{equation}
being $T>0$ and the data $f=f(x,t)$, $u_{0}=u_{0}(x)$ belonging to suitable functional spaces.\\

\noindent The first application of symmetrization techniques to linear Neumann elliptic problems goes back to the classical paper by Maderna-Salsa
\cite{MadSalsa}. Let us briefly describe the main result achieved in their paper. Let us consider a second order linear elliptic Neumann
problem of
the form
\begin{equation} \label{SalsaNeum}
\left\{
\begin{array}
[c]{lll}%
Lu=f   &  & in\text{ }%
\Omega,\\[10pt]
\dfrac{\partial u}{\partial\nu}=0 &  & on\text{ }\partial\Omega,
\end{array}
\right. %
\end{equation}
where
\begin{equation}
Lu=-\sum_{i,j} \partial_i(a_{ij}\partial_j u)\,,\label{ellipticop}
\end{equation}
posed in a smooth bounded domain $\Omega\subseteq \ren$;  the coefficients $\{a_{ij}\}$ are assumed to be bounded, measurable and satisfy the usual normalized
ellipticity condition
\begin{equation}
\sum_{i,j}a_{ij}(x)\xi_{i}\xi_{j}\geq|\xi|^{2}\quad \text{for a.e. }x\in\Omega,\,\forall\xi\in\R^{N};\label{elliptcond}
\end{equation}
finally, we impose the compatibility condition \eqref{compatib}. The absence of homogeneous Dirichlet boundary condition prevents us from using some features of the classical
analysis
introduced by Talenti \cite{Talenti1}, leading to pointwise comparison between the symmetrized version of the actual solution of the problem $u(x)$ and
the radially symmetric solution $v(|x|)$ of some radially symmetric model problem which is posed in a ball with the same volume as $\Omega$: indeed, a
rough
explanation of such an issue is that the Neumann boundary conditions imply that the level sets $\left\{x\in\Omega:|u(x)|>t\right\}$ of a solution $u$ are not
compactly contained in $\Omega$ (as for the zero Dirichlet data case), thus a part of the boundary of such sets may be contained in $\partial\Omega$. This
forces the
use of the \emph{relative} isoperimetric inequality, saying that for any measurable subset $E$ of $\Omega$ one has
\begin{equation}
\left[\min\left\{|E|,|\Omega\setminus E|\right\}\right]^{1-\frac{1}{N}}\leq Q P_{\Omega}(E)\label{relisop}
\end{equation}
being $\left\vert
E\right\vert $ the $N$-dimensional Lebesgue measure of $E$, $P_{\Omega}(E)$ the perimeter (in the De Giorgi sense) of $E$ in $\Omega$ and  the best value
of the
constant $Q$ in \eqref{relisop} depends on $\Omega$. Then, the choice of the classical truncation functions introduced in \cite{Talenti1} as test functions in the weak formulation
of
\eqref{SalsaNeum} leads to the issue of choosing the minimum value in \eqref{relisop}, being $E$ the level set $\left\{x\in\Omega:|u(x)|>t\right\}$ for
some
$t\geq0$. A key role for solving such problem is played by the so called \emph{median} $\mathsf{m}(u)$ defined by
\begin{equation}
\mathsf{m}(u)=\inf\left\{k\in\R:|\left\{x\in\Omega:u(x)>k\right\}|\leq |\Omega|/2\right\}:\label{median}
\end{equation}
indeed, the function $u_{1}=u-\mathsf{m}(u)$ is still a solution to the homogeneous Neumann problem \eqref{SalsaNeum}, and its remarkable property relies in
the
fact that
\[
|\text{sprt }u_{1}^{+}|\leq|\Omega|/2\,\quad |\text{sprt }u_{1}^{-}|\leq|\Omega|/2,
\]
where $u_{1}^{+}$, $u_{1}^{-}$ are the positive and the negative parts of $u_{1}$ respectively. Then testing the first equation in \eqref{SalsaNeum} with
the test
functions of \cite{Talenti1} constructed on the level sets of $u_{1}^{\pm}$ gives that the minimum in \eqref{relisop} is achieved by $|E|$, where
$E=\left\{x\in\Omega: u_{1}^{\pm}(x)>t\right\}$: then the classical method shown in \cite{Talenti1} takes naturally to the following two pointwise estimates in the
ball
$B$ centered at the origin such that $|B|=|\Omega|/2$:
\begin{equation}
(u_{1}^{+})^{\#}(x)\leq v_{1}(x)\label{firstpointest}
\end{equation}
\begin{equation}
(u_{1}^{-})^{\#}(x)\leq v_{2}(x),\label{secondpointest}
\end{equation}
 where $(u_{1}^{\pm})^{\#}$ is the \emph{Schwarz} decreasing rearrangement of $(u_{1}^{\pm})$ (see Section 2 for its precise definition and related
 properties)
 and $v_{i}$, $i=1,2$ is the solution of the \emph{Dirichlet} problem
\begin{equation*} \label{DirichNeum}
\left\{
\begin{array}
[c]{lll}%
-\gamma\Delta v_{i}=f_{i}^{\#}\left(  x\right)   &  & in\text{ }%
B,\\[10pt]
v_{i}=0 &  & on\text{ }\partial B,
\end{array}
\right. %
\end{equation*}
where $f_{1}$, $f_{2}$ are the positive and negative part of $f$ respectively and $\gamma=1/(N\omega_{N}^{1/N}Q)^{2}$, with $\omega_{N}$ being the measure of the unit ball in $\R^{N}$, which implies that $N\omega_{N}^{1/N}$
is the best constant in the classical isoperimetric inequality.
The meaningful result represented by estimates \eqref{firstpointest}-\eqref{secondpointest} shows on one hand that there is no ``worst'' problem
(unlike the
Dirichlet case) among the class of Neumann problems, defined by fixing the distribution function of $f$ and the measure of the ground domain, which problem
\eqref{SalsaNeum} could be compared to, in the sense of a pointwise comparison; on the other hand, the same inequalities  \eqref{firstpointest}-\eqref{secondpointest}, along with the estimates with
sharp
constants obtained for the Dirichlet problem in \cite{Talenti1}, easily allow to derive some optimal regularity estimates for the original solution $u$ in
terms
of the data $f$.\\

\noindent Since the paper \cite{MadSalsa}, many other works dealing with symmetrization in Neumann problems enriched the already existing literature with
interesting developments and sharper results: among them, it is worth mentioning the contributions given in \cite{AlTrMat}, \cite{Ferone},
\cite{BettaNeum},
\cite{FeroneMercaldoStein} for the linear case, in the paper \cite{AlCianMAz} dealing with the nonlinear elliptic framework. Furthermore,
we
refer to \cite{Bramanti} for an interesting treatment of the linear parabolic case.

\noindent Proceeding as in the papers \cite{dBVol}, \cite{VazVol1}, \cite{VazVol2}, \cite{VazVolSire}, due to the nonlocal nature of problem \eqref{eq.0},
it
will be essential to link such a problem to a suitable, local \emph{extension problem}, whose solution $w(x,y)$, a \emph{harmonic extension} of $u$, is
defined
on the infinite cylinder $\mathcal{C}_{\Omega}=\Omega\times(0,\infty)$, to which classical symmetrization techniques (with respect to the $x\in\Omega$-
variable)
can be applied: the issues arising in this approach will be mainly due to the Neumann boundary conditions and the presence of the ``extra'' variable
$y\geq0$,
which is fixed in the symmetrization arguments, an important detail implying that reaching a pointwise comparison is hopeless. Then an integral (or mass
concentration) comparison is expected, and since $u$ is the trace of $w$ over $\Omega\times\left\{0\right\}$, once we obtain a comparison result for the
extension $w$ of $u$, an estimate for $u$ is immediately derived.\\[1pt]
\noindent We notice that problem \eqref{eq.0} can be rewritten as
\[
(-\Delta_{N})^{\sigma}u+cu=f
\]
where the operator $(-\Delta_{N})^{\sigma}$ is the so called \emph{spectral Neumann fractional Laplacian}, whose definition and domain encode in particular
the
homogeneous Neumann boundary conditions: then the existence of an extension problem, associated to $(-\Delta_{N})^{\sigma}$,  with zero Neumann boundary
conditions
on the lateral surface $\partial_{L}\mathcal{C}_{\Omega}$ of the cylinder $\mathcal{C}_{\Omega}$, follows by
\cite[Theorem 1.1]{Stinga-Torrea},
generalizing the by now classical result by Caffarelli and Silvestre \cite{Caffarelli-Silvestre}. We also refer to \cite{Gale-Miana-Stinga} for a result of this nature in an even more general setting.

\medskip

\noindent {\sc Organization of the paper and main results.} Section \ref{Sec2} contains the preliminaries about symmetrization and mass
concentration that we will largely use throughout the paper. In Section \ref{Sec3} we give all the necessary functional background related to problem \eqref{eq.0}, which is
naturally connected to the very definition of the
operator $(-\Delta_{N})^{\sigma}$. Section \ref{Sec4} is entirely devoted to the introduction and the proof of the main result, consisting in comparing the solution
$u$ to
\eqref{eq.0} with $c=0$ to the following Dirichlet radial problem
\begin{equation*} \label{symmetriz}
\left\{
\begin{array}
[c]{lll}%
\left(  -\gamma\Delta\right)^{\sigma}v=f_{1}^{\#}(x)+f_{2}^{\#}\left(  x\right)   &  & in\text{ }%
B,\\ [10pt]
v=0  &  & on\text{ }\partial B
\end{array}
\right. %
\end{equation*}
where the operator $(-\Delta)^{\sigma}$ is the so called \emph{spectral Dirichlet fractional Laplacian} $(-\Delta_{D})^{\sigma}$, and  $f_{1}$, $f_{2}$ are the
positive and negative part of $f$ respectively. Moreover, making use of some results of \cite{dBVol}, a number of important
regularity estimates of $u$ in terms of the data $f$ are then derived. Section \ref{zeroordersec} provides the generalization of the comparison result shown in Section \ref{Sec4}
for problems appearing in the form
\eqref{eq.0} with a positive constant $c$. This last result is applied in Section \ref{parabolic} in the iterations of the parabolic implicit time discretization scheme, that
allows to establish an interesting concentration comparison for solutions of linear parabolic problems with Neumann boundary conditions of the form \eqref{linearparabintr}.

%%%%%%%%%%%%%%%%%%%%%%%%%%%%%%%%%%%%%%%%%%%%%%%%%%%%%%%%%%%%%%%%%%%%%%%%%%%%%%%%%%%%%%%%%%%%%%%%%%%%%%%%%%%%%%%%%%%%%%%%%%%%%%%%%%%%%%%%%%%%%%%%%%%%%%%%%%%%%%%%%

\section{On symmetrization and related properties} \label{Sec2}
In this Section we briefly recall the basic notions of Schwarz symmetrization and some related fundamental properties. Readers who are interested in more
details
of the theory are warmly addressed to the classical monographs \cite{Hardy}, \cite{Bennett},  \cite{Kesavan}, \cite{Bandle} or in the paper
\cite{Talentirearrinv}.\\

\noindent A measurable real function $f$ defined on $\R^{N}$ is called \emph{radially symmetric} (or \emph{radial}) if there is a function
$\widetilde{f}:[0,\infty)\rightarrow \R$ such that $f(x)=\widetilde{f}(|x|)$ for all $x\in \R^{N}$. We will often write $f(x)=f(r)$,
$r=|x|\ge0$ for such functions by abuse of notation. We say that $f$ is \emph{rearranged} if it is radial, nonnegative and $\widetilde{f}$ is a
right-continuous, non-increasing function of $r>0$. A similar definition can be applied for real functions defined on a ball
$B_{R}(0)=\left\{x\in\R^{N}:|x|<R\right\}$.

Now, let $\Omega$ be an open set of $%
%TCIMACRO{\U{211d} }%
%BeginExpansion
\mathbb{R}
%EndExpansion
^{N}$ and $f$ be a real measurable function on $\Omega$. We then define the
\emph{distribution function} $\mu_{f}$ of $f$ as%
\[
\mu_{f}\left(  k\right)  =\left\vert \left\{  x\in\Omega:\left\vert f\left(
x\right)  \right\vert >k\right\}  \right\vert \text{ , }k\geq0,
\]
and the \emph{one dimensional decreasing rearrangement} of $f$ as%
\[
f^{\ast}\left(  s\right)  =\sup\left\{ k\geq0:\mu_{f}\left(  k\right)
>s\right\}  \text{ , }s\in\left(  0,\left\vert \Omega\right\vert \right).
\]
We may also  think of extending $f^{\ast}$ as the  zero function in $[|\Omega|,\infty)$ if $\Omega$ is bounded. From this definition it follows
that $\mu_{f^{\ast}}=\mu_{f}$ (\emph{i.\,e.\,,} $f$, and $f^{\ast}$ are equi-distributed) and $f^{\ast}$ is exactly  the \emph{generalized right
inverse function} of
$\mu_{f}$.
Furthermore, if $\Omega^{\#}$ is the ball of $%
%TCIMACRO{\U{211d} }%
%BeginExpansion
\mathbb{R}
%EndExpansion
^{N}$ centered at the origin having the same Lebesgue measure as $\Omega,$ we define the
function
\[
f^{\#}\left(  x\right)  =f^{\ast}(\omega_{N}\left\vert x\right\vert
^{N})\text{ \ , }x\in\Omega^{\#},
\]
that will be called \emph{spherical decreasing rearrangement}, or \emph{Schwarz decreasing rearrangement}, of $f$. We easily infer that $f$ is rearranged
if and
only if
$f=f^{\#}$.

The only properties which will turn useful for what follows are
the conservation of the
$L^{p}$
norms: for all $p\in[1,\infty]$
\[
\|f\|_{L^{p}(\Omega)}=\|f^{\ast}\|_{L^{p}(0,|\Omega|)}=\|f^{\#}\|_{L^{p}(\Omega^{\#})}\,,
\]
as well as the classical Hardy-Littlewood inequality
\begin{equation}
\int_{\Omega}\left\vert f\left(  x\right)  g\left(
x\right)  \right\vert dx\leq\int_{0}^{\left\vert \Omega\right\vert }f^{\ast
}\left(  s\right)  g^{\ast}\left(  s\right)  ds=\int_{\Omega^{\#}}f^{\#}(x)\,g^{\#}(x)\,dx\,,
\label{HardyLit}%
\end{equation}
where $f,g$ are measurable functions on $\Omega$.

\noindent $\bullet$ We will often deal with two-variable functions of the type
\begin{equation}\label{f}%
f:\left(  x,y\right)  \in\mathcal{C}_{\Omega}\rightarrow f\left(  x,y\right)
\in{\mathbb{R}}
\end{equation}
defined on the cylinder  $\mathcal{C}%
_{\Omega}:=\Omega\times\left(  0,+\infty\right)  $, and measurable with respect to
$x.$ In that case, it will be convenient to define the so-called {\sl Steiner symmetrization} of $\mathcal{C}_{\Omega}$ with
respect to the variable $x$, namely the set \ \hbox{$\mathcal{C}_{\Omega}^{\#}:=\Omega
^{\#}\times\left(  0,+\infty\right).$} In addition, we will denote by $\mu
_{f}\left(  k,y\right)  $ and $f^{\ast}\left(  s,y\right)  $ the distribution
function and the decreasing rearrangement of (\ref{f}), with respect to $x$
for $y$ fixed, and we will also define the function%
\[
f^{\#}\left(  x,y\right)  =f^{\ast}(\omega_{N}|x|^{N},y)
\]
which is called the \emph{Steiner symmetrization} of $f$, with respect to the line
$x=0.$ Clearly, $f^{\#}$ is a spherically symmetric and decreasing function
with respect to $x$, for any fixed $y$.

\noindent $\bullet$ We recall now two important differentiation formulas that will prominently come into play in our arguments. They are basically used
when
dealing with the derivation of sharp estimates for the rearrangement $u^{\ast}$ of a solution $u$ to a certain parabolic problem, since in that context
it is essential to differentiate with respect to the extra variable $y$ under the integral symbol, for functions defined in the form
\[
\int_{\left\{x:u(x,y)>u^{*}(s,y)\right\}}\frac{\partial u}{\partial y}(x,y)\,dx\,.
\]
Here we recall two formulas, of first and second order, available in literature. The following proposition can be
found
in
\cite{Mossino}, and is a generalization of a well-known result by C. Bandle (see \cite{Band2}).
\begin{proposition}\label{BANDLE}
Suppose that $f\in H^{1}(0,T;L^{2}(\Omega))$ is nonnegative, for some $T>0$. Then $$f^{*}\in H^{1}(0,T;L^{2}(0,|\Omega|))$$ and if \
$|\left\{x:f(x,t)=f^{*}(s,t)\right\}|=0$ \
for
a.e. $(s,t)\in(0,|\Omega|)\times(0,T)$, the following differentiation formula holds:
\begin{equation}
\int_{\left\{x:f(x,y)>f^{*}(s,y)\right\}}\frac{\partial f}{\partial y}(x,y)\,dx=\int_{0}^{s}\frac{\partial f^{*}}{\partial y}(\tau,y)\,d\tau.\label{Rakotoson}
\end{equation}
\end{proposition}

Moreover, the following second order differentiation formula (which was also proved in \cite{AlLionTrom} in a more regular framework) is due to Mercaldo and
Ferone
(see \cite{FeroneMercaldoStein}):

\begin{proposition}
 Let us choose a nonnegative function $f\in W^{2,\infty}\left(  \mathcal{C}_{\Omega}\right)  $. Then for almost every
$y\in(0,+\infty)$
the following
 differentiation formula holds:
\begin{align}
\int_{\left\{x:f\left(  x,y\right)  >f^{\ast}\left(  s,y\right)\right\}  }\frac{\partial^{2}%
f}{\partial y^{2}}\left(  x,y\right)  dx  &  =\frac{\partial^{2}}{\partial
y^{2}}\int_{0}^{s}f^{\ast}\left(  \tau,y\right)  d\tau-\int_{\left\{x:f\left(
x,y\right)  =f^{\ast}\left(  s,y\right)\right\}  }\frac{\left(  \frac{\partial
f}{\partial y}\left(  x,y\right)  \right)  ^{2}}{\left\vert \nabla
_{x}f\right\vert }\,d\mathcal{H}^{N-1}\left(  x\right)
 \nonumber\\
&  \!\!\!+\left(  \int_{\left\{x:f\left(  x,y\right)  =f^{\ast}\left(  s,y\right)\right\}}
\!\frac{\frac{\partial f}{\partial y}\left(  x,y\right)  }{\left\vert
\nabla_{x}f\right\vert }\,d\mathcal{H}^{N-1}\left(  x\right)  \!\right)
^{2}\!\left(  \!\int_{\left\{x:f\left(  x,y\right)  =f^{\ast}\left(  s,y\right)\right\}
}\!\frac{1}{\left\vert \nabla_{x}f\right\vert }\,d\mathcal{H}^{N-1}\left(
x\right)  \!\right)  ^{-1}\!
.\label{Ferone-Mercaldo}
\end{align}
\end{proposition}

%%%%%%%%%%%%%%%%%%%%%%%%%%%%%%%%%%%%%%%%%%%%%%%%%%%%%%%%%%%%%%%%%%%%%%%%%%%%%%%%%%%%%%%%%%%%%%%%%%%%%%%%%%%%%%%%%%%%%%%%%%%%%%%%%%%%%%%%%%%%%%%%%%%%%%%%%%%%

\subsection{Mass concentration}

Since we will provide estimates of the solutions of our fractional elliptic and parabolic problems in terms of their integrals, the following
definition, taken from  \cite{Vsym82}, is of basic importance.

\begin{definition}
Let $f,g\in L^{1}_{loc}(\R^{N})$ be two nonnegative radially symmetric functions on $\R^{N}$. We say that $f$ is less concentrated than $g$, and we write
$f\prec
g$ if for
all $R>0$ we get
\[
\int_{B_{R}(0)}f(x)dx\leq \int_{B_{R}(0)}g(x)dx.
\]
\end{definition}
The partial order relationship $\prec$ is called \emph{comparison of mass concentrations}.
Of course, this definition can be suitably adapted if $f,g$ are radially symmetric and locally integrable functions on a ball $B_{R}$. Moreover, we have
that
$f\prec g$ if and only if
\[
\int_{0}^{s}f^{\ast}(\tau)d\tau\leq \int_{0}^{s}g^{\ast}(\tau)d\tau,
\]
for all $s\geq0$.

The comparison of mass concentrations enjoys a nice equivalent formulation if $f$ and $g$ are rearranged. Indeed the following result holds (for the proof we refer to
\cite{Chong}, \cite{VANS05}):

\begin{lemma}\label{lemma1}
Let $f,g\in L^{1}(\Omega)$ be two rearranged functions on a ball $\Omega=B_{R}(0)$. Then $f\prec g$ if and only if for every convex
nondecreasing
function
$\Phi:[0,\infty)\rightarrow [0,\infty)$ with $\Phi(0)=0$ we have
\begin{equation}
\int_{\Omega}\Phi(f(x))\,dx\leq \int_{\Omega}\Phi(g(x))\,dx.
\end{equation}
This result still holds if $R=\infty$ and $f,g\in L^{1}_{loc}(\R^{N})$ with $g\rightarrow0$ as $|x|\rightarrow\infty$, in the sense that $\mu_{g}(k)<\infty$ for all $k>0$.
\end{lemma}
From this Lemma it easily follows that if $f\prec g$ and $f,g$ are rearranged, then
\begin{equation}
\|f\|_{L^{p}(\Omega)}\leq \|g\|_{L^{p}(\Omega)}\quad \forall p\in[1,\infty].
\end{equation}

%%%%%%%%%%%%%%%%%%%%%%%%%%%%%%%%%%%%%%%%%%%%%%%%%%%%%%%%%%%%%%%%%%%%%%%%%%%%%%%%%%%%%%%%%%%%%%%%%%%%%%%%%%%%%%%%%%%%%%%%%%%%%%%%%%%%%%%%%%%%%%%%%%%%%%%%%%%%%%
\section{Functional background}\label{Sec3}
In this Section we provide a self-contained description of the functional background which is necessary for the well-posedness of problems of the
type
\eqref{eq.0}. Most of the material we present here are excerpts of the papers \cite{StingaVolz}, \cite{PellacciMontef} to which we refer the interested
reader
for extra details. Furthermore, we point out that another version of a nonlocal elliptic Neumann problem is available in literature, see \emph{e.g.} \cite{ValdinoRos}.

Let us consider the homogeneous Neumann eigenvalue problem for the Laplacian on a smooth bounded domain $\Omega$ of $\R^{N}$:
\begin{equation} \label{eigen}
\left\{
\begin{array}
[c]{lll}%
-\Delta\varphi=\lambda\varphi   &  & in\text{ }%
\Omega,\\[10pt]
\dfrac{\partial \varphi}{{\partial\nu}}=0 &  & on\text{ }\partial\Omega.
\end{array}
\right. %
\end{equation}
It is well known (see for example \cite{Evans,Gilbarg-Trudinger}) that there exists a sequence of nonnegative eigenvalues
$\{\lambda_{k}\}_{k\in\N_{0}}$ corresponding to eigenfunctions $\{\varphi_{k}\}_{k\in\N_{0}}$ in $H^{1}(\Omega)$, the latter
being weak solutions to \eqref{eigen}. We have that $\lambda_{0}=0$, $\varphi_{0}=1/\sqrt{|\Omega|}$,
$\int_{\Omega}\varphi_{k}\, dx=0$ and $\varphi_k$ belongs to $C^\infty(\overline{\Omega})$ for all $k\geq1$.

In order to introduce the spectral Neumann fractional Laplacian
$(-\Delta_N)^{\sigma}$, we define its domain as
$$\mathcal{H}^{\sigma}(\Omega)\equiv\Dom((-\Delta_N)^{\sigma}):=
\Big\{u\in L^2(\Omega):\sum_{k=1}^\infty\lambda_k^{\sigma}\,|\langle u,\varphi_k
\rangle_{L^2(\Omega)}|^2<\infty\Big\},$$
which is a Hilbert space equipped with the scalar product
\[
\langle u,v\rangle_{\mathcal{H}^{\sigma}(\Omega)}:=\langle u,v\rangle_{L^{2}(\Omega)}+\sum_{k=1}^{\infty}\lambda_{k}^{\sigma}\, \langle
u,\varphi_k\rangle_{L^2(\Omega)}\langle
v,\varphi_k\rangle_{L^2(\Omega)},
\]
defining the following norm in $\mathcal{H}^{\sigma}(\Omega)$:
$$\|u\|^2_{\mathcal{H}^{\sigma}(\Omega)}=\|u\|_{L^2(\Omega)}^{2}+\sum_{k=1}^{\infty}\lambda_{k}^{\sigma}\,|\langle
u,\varphi_k\rangle_{L^2(\Omega)}|^{2}.$$
Since $\lambda_k\nearrow\infty$, it is obvious that $C^\infty(\overline{\Omega})\subset H^1(\Omega)\subset\mathcal{H}^{\sigma}(\Omega)$.
For $u\in\mathcal{H}^{\sigma}(\Omega)$, we define $(-\Delta_N)^{\sigma}u$ as {an} element in the dual space
$\mathcal{H}^{\sigma}(\Omega)'$ given by
$$(-\Delta_N)^{\sigma}u=\sum_{k=1}^\infty\lambda_k^{\sigma}\,\langle u,\varphi_k\rangle_{L^2(\Omega)}\varphi_k,
\quad\hbox{in}~\mathcal{H}^{\sigma}(\Omega)',$$
that is, for any function $v\in\mathcal{H}^{\sigma}(\Omega)$,
$$\langle(-\Delta_N)^{\sigma}u,v\rangle=\sum_{k=1}^\infty\lambda_k^{\sigma}\,\langle u,\varphi_k\rangle_{L^2(\Omega)}
\langle v,\varphi_k\rangle_{L^2(\Omega)}.$$
Notice that the set of constant functions
is the nontrivial kernel {of
$(-\Delta_N)^{\sigma}$ in $\mathcal{H}^{\sigma}(\Omega)$}.
The last identity can be rewritten as
$$\langle(-\Delta_N)^{\sigma}u,v\rangle=\int_\Omega\left[(-\Delta_N)^{\sigma/2}u\right]\left[(-\Delta_N)^{\sigma/2}v\right]\,dx,\quad\hbox{for
every}~v\in\mathcal{H}^{\sigma}(\Omega),$$
where $(-\Delta_{{N}})^{\sigma/2}$ is defined by taking the power $\sigma/2$ of the eigenvalues $\lambda_k$.\\
Actually it is possible to identify (see \cite[Theorem 2.4]{StingaVolz} and the generalization in \cite{Caffarelli-Stinga}) the domain $\mathcal{H}^{\sigma}(\Omega)$ of $(-\Delta_N)^{\sigma}$ with the fractional
Sobolev space $H^{\sigma}(\Omega)$.

Now let us consider problem \eqref{eq.0} with $c=0$, that is the problem
\begin{equation} \label{eq.0.1}
\left\{
\begin{array}
[c]{lll}%
\left(  -\Delta\right)^{\sigma}u=f\left(  x\right)   &  & in\text{ }%
\Omega,\\[10pt]
\dfrac{\partial u}{{\partial\nu}}=0 &  & on\text{ }\partial\Omega.
\end{array}
\right. %
\end{equation}
which in our notations can be written in the form
\[
(-\Delta_{N})^{\sigma}u=f.
\]
For a function $u\in L^{1}(\Omega)$ we set
\[
u_{\Omega}:=\frac{1}{|\Omega|}\int_{\Omega}u\,dx.
\]
We define the functional spaces
\[
\mathfrak{H}^{\sigma}(\Omega):=\left\{u\in H^{\sigma}(\Omega):u_{\Omega}=0\right\}
\]
and
\[
\mathscr{H}^{1}(\mathcal{C}_{\Omega},y^{1-2\sigma}):=\left\{w\in H^{1}(\mathcal{C}_{\Omega},y^{1-2\sigma}):(w(\cdot,y))_{\Omega}=0,\,\forall y\geq0\right\},
\]
where $H^{1}(\mathcal{C}_{\Omega},y^{1-2\sigma})$ is the weighted Sobolev space with respect to the weight $y^{1-2\sigma}$.
By \cite[Lemma 2.2]{PellacciMontef}) the space $\mathscr{H}^{1}(\mathcal{C}_{\Omega},y^{1-2\sigma})$ can be equipped with the norm
\[
\|w\|_{\mathscr{H}^{1}(\mathcal{C}_{\Omega},y^{1-2\sigma})}=\left(\iint_{\mathcal{C}_{\Omega}}y^{1-2\sigma}|\nabla_{x,y}w|^{2}dx\,dy\right)^{1/2}.
\]
It is possible to provide the following useful characterization of $\mathfrak{H}^{\sigma}(\Omega)$ (see \cite[Proposition 2]{PellacciMontef}):
\begin{align*}
\mathfrak{H}^{\sigma}(\Omega)&=\left\{u=\tr w:w\in \mathscr{H}^{1}(\mathcal{C}_{\Omega},y^{1-2\sigma}) \right\}\\
&=\Big\{u\in L^2(\Omega):u_{\Omega}=0\,and\,\sum_{k=1}^\infty\lambda_k^{\sigma}|\langle u,\varphi_k
\rangle_{L^2(\Omega)}|^2<\infty\Big\};
\end{align*}
moreover, $\mathfrak{H}^{\sigma}(\Omega)$ is an Hilbert space equipped with the Gagliardo seminorm
\[
\|u\|_{\mathfrak{H}^{\sigma}(\Omega)}:=[u]_{H^{\sigma}(\Omega)}:=\left(\int_{\Omega}\int_{\Omega}\frac{|u(x)-u(y)|^{2}}{|x-y|^{n+2\sigma}}\,dx\,dy\)^{1/2}.
\]
Now we wish to particularize the general extension problem proved in
\cite{Stinga-Torrea} for the operator $(-\Delta_N)^{\sigma}$ restricted to $\mathfrak{H}^{\sigma}(\Omega)$ .

\begin{theorem}[Extension problem for $(-\Delta_N)^{\sigma}$]\label{extensth}
Let $u\in \mathfrak{H}^{\sigma}(\Omega)$. Define
\begin{equation}w(x,y):=\sum_{k=1}^{\infty}\rho(\lambda_{k}^{1/2}y)\,
\langle u,\varphi_k\rangle_{L^2(\Omega)}\varphi_{k}(x),\label{trueextens}\end{equation}
where the function $\rho(t)$ solves the problem
\begin{equation} \label{ODE}
\left\{
\begin{array}
[c]{lll}%
\rho^{\prime\prime}(t)+\frac{1-2\sigma}{t}\rho^{\prime}(t)=\rho(t)   &  & t>0,\\[5pt]
-\lim_{t\rightarrow 0^{+}}t^{1-2\sigma}\rho^{\prime}(t)=\kappa_{\sigma},\\[5pt]
\rho(0)=1,
\end{array}
\right. %
\end{equation}
where
\[
\kappa_{\sigma}:=\frac{2^{1-2\sigma}\,\Gamma(1-\sigma)}{\Gamma(\sigma)}.
\]
Then $w\in \mathscr{H}^{1}(\mathcal{C}_{\Omega},y^{1-2\sigma})$ and it is the unique weak solution to the extension problem
\begin{equation}\label{extens1}
\begin{cases}
\dfrac{\partial^{2}w}{\partial y^{2}}+\dfrac{1-2\sigma}{y}\dfrac{\partial w}{\partial y}+\Delta_{x}w=0,&\hbox{in}~\mathcal{C}_{\Omega},\\
\\
 \dfrac{\partial w}{\partial \nu_{x}}=0,&\hbox{on}~\partial_L\mathcal{C}_{\Omega},\\
\\
 w(x,0)=u(x),&\hbox{on}~\Omega,
\end{cases}
\end{equation}
where $\nu$ is the outward normal
to the lateral boundary $\partial_L\mathcal{C}_{\Omega}$ of $\mathcal{C}_{\Omega}$.
More precisely,
$$\iint_{\mathcal{C}_{\Omega}}y^{1-2\sigma}\big(\nabla_{x}w\cdot\nabla_{x}\psi+ w_y\psi_y\big) \,dx\,dy=0,$$
for all test functions $\psi\in H^1(\mathcal{C}_{\Omega},y^{1-2\sigma})$ with zero trace over $\Omega$, i.e. $\tr_\Omega\psi=0$,
and $\lim_{y\to0^+}w(x,y)=u(x)$ in $L^2(\Omega)$. Furthermore, the function $w$ is the unique minimizer of the energy functional
\begin{equation}\mathcal{F}(w)=\frac{1}{2}\iint_{\mathcal{C}_{\Omega}}y^{1-2\sigma}\big(|\nabla_{x}w|^{2}+|w_{y}|^{2}\big)\,dx\,dy,\label{energufunct}\end{equation}
over the set  $\mathcal{U}=\left\{w\in H^{1}(\mathcal{C}_{\Omega},y^{1-2\sigma}):\,\tr_{\Omega}w=u\right\}$. We can also write
\begin{equation}\label{eq.5}
w(x,y)=\frac{y^{2\sigma}}{4^{\sigma}\Gamma(\sigma)}\int_{0}^{\infty}e^{- y^2/(4t)}e^{t\Delta_N}u(x)\,\frac{dt}{t^{1+\sigma}}
\end{equation}
where $e^{t\Delta_N}u(x)$ is the heat diffusion semigroup generated by the Neumann Laplacian acting on $u$.
An equivalent formula for $w$ is
$$w(\cdot,y)=\frac{1}{\Gamma(\sigma)}\int_0^\infty e^{- y^2/(4t)}e^{t\Delta_N}((-\Delta_N)^{\sigma}u)\,\frac{dt}{t^{1-\sigma}}.$$
Moreover,
\begin{equation}\label{DirchNeuma}
-\frac{1}{\kappa_{\sigma}}\lim_{y\to 0^+}y^{1-2\sigma}w_y=(-\Delta_N)^{\sigma}u,\quad\hbox{in}~\mathcal{H}^{\sigma}(\Omega)'.\end{equation}
\end{theorem}
For the proof of Theorem \ref{extensth} see \cite[Theorem 2.1]{StingaVolz} and \cite[Theorem 1.1]{Stinga-Torrea}. We also notice that the solution $\rho(y)$ to problem \eqref{ODE} is explicit and given in terms of modified Bessel of the third kind (see \cite[Section 3.1]{Stinga-Torrea}. \\
For any $u\in \mathfrak{H}^{\sigma}(\Omega)$, we will call the solution $w$ to problem \eqref{extens1}
the \emph{Neumann harmonic extension of $u$} and we write $w=E(u)$.

With this definition at hand, we assume that $f\in \mathfrak{H}^{-\sigma}(\Omega):=\left\{g\in H^{\sigma}(\Omega)^{\prime}:\langle g,1\rangle=0\right\}$: if $f\in
L^{2}(\Omega)$, this condition imposes $f$ to satisfy the compatibility condition \eqref{compatib}; furthermore  it can be proved (see \cite[Proposition
3]{PellacciMontef}) that $\mathfrak{H}^{-\sigma}(\Omega)$ is actually isomorphic to the dual space $(\mathfrak{H}^{\sigma}(\Omega))^{\prime}$ of
$\mathfrak{H}^{\sigma}(\Omega)$. Let us now consider problem \eqref{eq.0.1}, to which we associate the following extension problem
\begin{equation}
\left\{
\begin{array}
[c]{lll}%
\dfrac{\partial^{2}w}{\partial y^{2}}+\dfrac{1-2\sigma}{y}\dfrac{\partial w}{\partial y}+\Delta_{x}w=0  &  & in\text{
}\mathcal{C}_{\Omega},\\[15pt]
\ \dfrac{\partial w}{\partial \nu_{x}}=0 &  & on\text{ }\partial_{L}\mathcal{C}_{\Omega},\\[15pt]
\displaystyle{-\frac{1}{\kappa_{\sigma}}\lim_{y\rightarrow0^{+}}y^{1-2\sigma}\dfrac{\partial w}{\partial y}(x,y)}=f\left(  x\right)   &
&
in\text{ }\Omega.
\end{array}
\right.  \label{extens}%
\end{equation}
We give now the following, suitable definition of weak solution of problem \eqref{extens}:
\begin{definition}\label{weakdef1}
We say that a function $w\in
\mathscr{H}^{1}(\mathcal{C}_{\Omega},y^{1-2\sigma})$ is a weak solution to \eqref{extens} if
\begin{equation}\iint_{\mathcal{C}_{\Omega}}y^{1-2\sigma}\big(\nabla_{x} w\cdot\nabla_{x} \psi+
 w_{y}\psi_{y}\big)\,dx\,dy=\kappa_{\sigma}\langle f,\tr_\Omega \psi\rangle,\label{weakform1}\end{equation}
 for every $\psi\in \mathscr{H}^{1}(\mathcal{C}_{\Omega},y^{1-2\sigma})$.
\end{definition}
By the classical Lax-Milgram Theorem, the existence and uniqueness of weak solution to \eqref{extens} in the sense of definition \ref{weakdef1} is
immediate and the solution $w$ is explicit (see \cite{StingaVolz}).

As a direct consequence, we have that
\begin{lemma}
Let $f\in\mathfrak{H}^{-\sigma}(\Omega)$ and assume that $w \in
\mathscr{H}^{1}(\mathcal{C}_{\Omega},y^{1-2\sigma})$ is the weak solution to \eqref{extens}. Then $w$ takes the form \eqref{trueextens} and $u:=\tr w\in \mathfrak{H}^{\sigma}(\Omega)$ is the unique (weak) solution
in $\mathfrak{H}^{\sigma}(\Omega)$ to the linear problem \eqref{eq.0.1}.
\end{lemma}

Moreover, the space $\mathscr{H}^{1}(\mathcal{C}_{\Omega},y^{1-2\sigma})$ for test functions in
definition \ref{weakdef1} can be actually replaced by the whole space $H^{1}(\mathcal{C}_{\Omega},y^{1-2\sigma})$, as the following result shows (see \cite{PellacciMontef}):
\begin{proposition}\label{enlargementtestspace}
Let $f\in \mathfrak{H}^{-\sigma}(\Omega)$ and $w$ be the weak solution to problem \eqref{extens}. Then:
\smallskip

\noindent {\rm (i)} there exist positive constants $C,\,k$ such that for all $y>0$
\[
\int_{\Omega}|\nabla w(x,y)|^{2}dx\leq C\,e^{-ky};
\]
\noindent {\rm (ii)} equation \eqref{weakform1} holds for any  function $\psi\in H^{1}_{loc}(\mathcal{C}_{\Omega},y^{1-2\sigma})$, such that there is a positive constant $C$, uniform w.r. to $y>0$, for which
\[
\|\psi(\cdot,y)\|_{L^{2}(\Omega)}\leq C.
\]
\end{proposition}
\begin{remark}
\emph{Remark \ref{enlargementtestspace} allows to choose $\psi(x,y)=\eta(y)\theta(x)$, with $\eta\in C_{0}^{\infty}(0,\infty)$ and $\theta\in H^{1}(\Omega)$ as a test function in Definition \eqref{weakdef1}: therefore integrating by parts and using Fubini's Theorem we can conclude that (see also \cite{Caffarelli-Stinga})    \begin{equation}
\int_{\Omega}\nabla_{x} w(x,y)\cdot\nabla
\theta(x)\,dx=\int_{\Omega}\left(w_{yy}+\frac{1-2\sigma}{y}w_{y}\right)\theta(x)dx\label{weakformlevelbylevel}
\end{equation}
for all $\theta\in H_{1}(\Omega)$ and a.e. $y>0$. Moreover, since $w$ is defined by means of formula \eqref{trueextens} (or equivalently by \eqref{eq.5}), we also notice that $w$  is smooth in $\mathcal{C}_{\Omega}$ (see \cite{Stinga-Torrea}).}
\end{remark}

It is clear that if $u\in H^{\sigma}(\Omega)$ solves \eqref{eq.0.1}, then $u-u_{\Omega}$ is the unique solution to the same problem in the smaller space
$\mathfrak{H}^{\sigma}(\Omega)$: hence if $\overline{u}$ is the unique weak solution to \eqref{eq.0.1} in the space $\mathfrak{H}^{\sigma}(\Omega)$, all the
solutions in $H^{\sigma}(\Omega)$ to \eqref{eq.0.1} are of the form $u=\overline{u}+c$ for $c\in\R$.

As far as problem \eqref{eq.0} for $c>0$ is concerned, the space where to look for solutions of the extension problems is a direct generalization of the full description given in
\cite{StingaVolz} (see \cite{Caffarelli-Stinga}).
Indeed, if $u_{\Omega}\neq 0$ then the function
\begin{equation}\label{extnonzeraver}
w(x,y)=\sum_{k=0}^{\infty}\rho(\lambda_{k}^{1/2}y)\,
\langle u,\varphi_k\rangle_{L^2(\Omega)}\varphi_{k}(x)
\end{equation}
is in general not in $L^2(\mathcal{C}_{\Omega},y^{1-2\sigma})$ but only its gradient is (see for instance the computations in \cite[Lemma 4.3]{ColoradoBrandle}). Therefore, in order to give a suitable definition of the Neumann
harmonic extension of $u$, we first solve the extension problem \eqref{extens1}
with initial data $\tilde{u}=u-u_\Omega$, in order to find a function $\tilde{w}=E(\tilde{u})$. Then we define
$$w=E(u):=\tilde{w}+u_\Omega,$$
which clearly coincides with \eqref{extnonzeraver}.\\ Using the fact that the fractional Neumann Laplacian
does not see constants, we have $$(-\Delta_N)^{\sigma}u=(-\Delta_N)^{\sigma}\tilde{u}=
-\frac{1}{\kappa_{\sigma}}\lim_{y\rightarrow0^{+}}y^{1-2\sigma}\dfrac{\partial \tilde{w}}{\partial y}(x,y)=
--\frac{1}{\kappa_{\sigma}}\lim_{y\rightarrow0^{+}}y^{1-2\sigma}\dfrac{\partial w}{\partial y}(x,y),\quad\hbox{in}~\mathcal{H}^{\sigma}(\Omega)',$$
thus we recover the local interpretation \eqref{DirchNeuma} of the fractional Neumann Laplacian. \\
\noindent Since \eqref{extnonzeraver} is \emph{formally} a solution to \eqref{extens}, we have to define the right functional space where this extension
belongs to. Following \cite{StingaVolz}, we introduce the space $\mathsf{H}^{\sigma,c}(\mathcal{C}_{\Omega})$ as the completion of
$H^{1}(\mathcal{C}_{\Omega},y^{1-2\sigma})$ under the scalar product
\begin{equation*}\label{eq:dos estrellas}
(w,\psi)_{\sigma,c}=\iint_{\mathcal{C}_{\Omega}} y^{1-2\sigma}\nabla_{x,y} w\cdot\nabla_{x,y} \psi\,dx\,dy+c\kappa_{\sigma}\int_{\Omega}(\tr_{\Omega}w)(\tr_{\Omega}\psi)\,dx.
\end{equation*}
We denote by $\|\cdot\|_{\sigma,c}$ the associated norm:
\begin{equation*}
\|w\|_{\sigma,c}^2=\iint_{\mathcal{C}_{\Omega}}y^{1-2\sigma}|\nabla_{x,y} w|^2\,dx\,dy+c\kappa_{\sigma}\int_{\Omega}(\tr_{\Omega}w)^2\,dx.
 \end{equation*}
Notice that, for each $c>0$,
\[
H^{1}(\mathcal{C}_{\Omega},y^{1-2\sigma})\subset\mathsf{H}^{\sigma,c}(\mathcal{C}_{\Omega}),
\]
as Hilbert spaces, where the inclusion is strict, since constant functions belong to
$\mathsf{H}^{\sigma,c}(\mathcal{C})$ but not to $H^{1}(\mathcal{C}_{\Omega},y^{1-2\sigma})$. \\[0.5pt]
From \cite[Theorem 2.4, Lemma 2.5 ]{StingaVolz} it follows that a unique trace embedding from $\mathsf{H}^{\sigma,c}(\mathcal{C})$ to $H^{\sigma}(\Omega)$ can be
defined. Then we can give the following definition of weak solution for linear problems of the following form:
\begin{equation}
\left\{
\begin{array}
[c]{lll}%
\dfrac{\partial^{2}w}{\partial y^{2}}+\dfrac{1-2\sigma}{y}\dfrac{\partial w}{\partial y}+\Delta_{x}w=0  &  & in\text{
}\mathcal{C}_{\Omega},\\[15pt]
\ \dfrac{\partial w}{\partial \nu_{x}}=0 &  & on\text{ }\partial_{L}\mathcal{C}_{\Omega},\\[15pt]
\displaystyle{-\frac{1}{\kappa_{\sigma}}\lim_{y\rightarrow0^{+}}y^{1-2\sigma}\dfrac{\partial w}{\partial y}(x,y)}+c\,w(x,0)=f\left(  x\right)   &
&
in\text{ }\Omega:
\end{array}
\right.  \label{extenszero}%
\end{equation}

\begin{definition}
Let $f\in H^{\sigma}(\Omega)^{\prime}$. We say that a function $w\in
\mathsf{H}^{\sigma,c}(\mathcal{C})$ is a weak solution to \eqref{extenszero} if
\begin{equation*}(w,\psi)_{c,\sigma}=\kappa_{\sigma}\langle f,\tr_\Omega \psi\rangle,\label{eq.33}\end{equation*}
 for every $\psi\in \mathsf{H}^{\sigma,c}(\mathcal{C})$.
 \end{definition}
By the Lax-Milgram Theorem again we easily infer that a unique, explicit weak solution to \eqref{extenszero} exists (see \cite[Lemma 3.3]{StingaVolz}), and its trace $u=\tr w$ is the unique solution in
$H^{\sigma}(\Omega)$ to the linear problem \eqref{eq.0}. For further regularity properties concerning solutions to problems of the type \eqref{eq.0} we refer
the reader to \cite{StingaVolz} for the case $s=1/2$ and \cite{Caffarelli-Stinga} for a general exponent $\sigma\in(0,1)$.
\begin{remark}\emph{
In the radial Dirichlet problems appearing in our comparison theorems, namely the problems
\begin{equation*} \label{symmetriz}
\left\{
\begin{array}
[c]{lll}%
\left(  -\gamma\Delta\right)^{\sigma}v+cv=g(x)   &  & in\text{ }%
B,\\ [10pt]
v=0  &  & on\text{ }\partial B
\end{array}
\right. %
\end{equation*}
for some radial function $g$ and a positive constant $\gamma$, the operator $\left(  -\gamma\Delta\right)^{\sigma}$ is understood as the spectral
\emph{Dirichlet} fractional Laplacian $(-\gamma\Delta_{D})^{\sigma}$: for all the most useful properties of such operator we refer the interested reader to
the \cite{Cabre-Tan}, \cite{ColoradoBrandle}. \\ Finally, from now on we will always omit the subscripts in the powers of the Laplacian, since it will be always clear by the context which boundary conditions are chosen, so that the spectral definition of the operator changes accordingly.}
\end{remark}

%%%%%%%%%%%%%%%%%%%%%%%%%%%%%%%%%%%%%%%%%%%%%%%%%%%%%%%%%%%%%%%%%%%%%%%%%%%%%%%%%%%%%%%%%%%%%%%%%%%%%%%%%%%%%%%%%%%%%%%%%%%%%%%%%%%%%%%%%%%%%%%%%%%%%%%%%%%%%%%
\section{The main result}\label{Sec4}
The aim of this Section is to derive sharp estimates via symmetrization for solutions to fractional Neumann problems of the type \eqref{eq.0.1}. According to what
explained in the introduction, we will compare problem \eqref{eq.0.1} with the following fractional radially symmetric problem
\begin{equation} \label{symmetrizz}
\left\{
\begin{array}
[c]{lll}%
\left(  -\gamma\Delta\right)^{\sigma}v=f_{1}^{\#}(x)+f_{2}^{\#}\left(  x\right)   &  & in\text{ }%
B,\\ [10pt]
v=0  &  & on\text{ }\partial B,
\end{array}
\right. %
\end{equation}
where $f_{1}$, $f_{2}$ are the positive and negative part of $f$ respectively, $B$ is the ball centered at the origin with Lebesgue measure $|\Omega|/2$
and
$\gamma=1/(N\omega_{N}^{1/N}Q)^{2}$, being $N\omega_{N}^{1/N}$
and $Q$ the best constants in the isoperimetric and relative isoperimetric (see \eqref{relisop}) inequality respectively .
Then we associate to the solution $v$ to \eqref{symmetrizz} its \emph{Dirichlet} harmonic extension $\xi$, solving (see \cite[Theorem 1.1]{Stinga-Torrea})
\begin{equation}
\left\{
\begin{array}
[c]{lll}%
\dfrac{\partial^{2}\xi}{\partial y^{2}}+\dfrac{1-2\sigma}{y}\dfrac{\partial\xi}{\partial y}+\gamma\Delta_{x}\xi=0  &  & in\text{
}\mathcal{C}_{B},\\[15pt]
\xi=0 &  & on\text{ }\partial_{L}\mathcal{C}_{B},\\[15pt]
\displaystyle{-\frac{1}{\kappa_{\sigma}}\lim_{y\rightarrow0^{+}}y^{1-2\sigma}\dfrac{\partial\xi}{\partial y}(x,y)}=f_{1}^{\#}+f_{2}^{\#}&
&
in\text{ }B.
\end{array}
\right.  \label{extensionsymm}%
\end{equation}

According to \cite{Cabre-Tan},\cite{ColoradoBrandle} (see also the nice Appendix in \cite{SirBonfVaz})  we have that $v\in H(B)$, where
\begin{equation*}
H(B)=\left\{
\begin{array}
[c]{lll}%
H^{\sigma}(\Omega)&& if\,\sigma\in(0,1/2)\\
\\
H_{00}^{1/2}(\Omega) && if\,\sigma=1/2\\
\\
H_{0}^{\sigma}(\Omega) && if\,\sigma\in(1/2,1);
\end{array}
\right.
\end{equation*}
moreover the solution $\xi$ belongs to the energy space $X_{0}^{\sigma}(\mathcal{C}_{\Omega})$
defined as the completion of $C_{0}^{\infty}(\Omega\times[0,\infty))$ with respect to the norm
\begin{equation*}
\Vert \psi\Vert_{X_{0}^{\sigma}(\mathcal{C}_{\Omega})}:=\left(  \iint_{\mathcal{C}_{\Omega}}y^{1-2\sigma}|\nabla \psi(x,y)|^{2}\,dxdy\right)  ^{1/2}.
\end{equation*}

Our main goal is to compare any solution $u$ to \eqref{eq.0.1} with the solution $v$ to
\eqref{symmetrizz}. The most direct (and natural) way to proceed is to compare the Neumann extension $w$ of $u$, that is the
solution to \eqref{extens}, with the solution
$\xi$ to the extension problem \eqref{extensionsymm}. Before stating our main result, for all $y>0$ we define the function
\begin{equation}
\lambda(y)=\mathsf{m}((w(\cdot,y)))\label{functiont(y)}
\end{equation}
where $\mathsf{m}((w(\cdot,y)))$ is the median of the function $w(\cdot,y)$ (see \eqref{median}). Moreover, we set
\[
w_{1}(x,y)=[w(x,y)-\lambda(y)]^{+}\,\quad w_{2}(x,y)=[w(x,y)-\lambda(y)]^{-}.
\]
It is clear that, for all fixed $y\geq0$,
\[
|\left\{x\in\Omega:w(x,y)>\lambda(y)\right\}|\leq\frac{|\Omega|}{2}.
\]
and
\[
|\left\{x\in\Omega:w(x,y)\geq \lambda(y)\right\}|\geq\frac{|\Omega|}{2}
\]
thus
\begin{equation}
|\text{sprt}\,w_{i}(\cdot,y)|\leq |\Omega|/2,\label{support}
\end{equation}
for all $i=1,2$.\\
Then we can prove the following result
\begin{theorem}\label{comparisontheorem}
Let us choose a source term $f\in \mathfrak{H}^{-\sigma}(\Omega)$ and let $u\in H^{\sigma}(\Omega)$ be any solution to \eqref{eq.0.1}.
Assume that $w$ is the Neumann harmonic extension of $u$, namely the solution to the extension problem \eqref{extens} associated to \eqref{eq.0.1}. Let $v$ be the solution to \eqref{symmetrizz} and $\xi$ its Dirichlet harmonic extension, solving \eqref{extensionsymm}. Then for all
$y\geq0$, we have
\[
w_{1}^{\#}(\cdot,y)+w_{2}^{\#}(\cdot,y)\prec \xi(\cdot,y)
\]
that is
\begin{equation}
\int_{0}^{s}(w^{\ast}_{1}(\tau,y)+w^{\ast}_{2}(\tau,y))d\tau\leq\int_{0}^{s}\xi^{\ast}(\tau,y)d\tau\label{concentrestim}
\end{equation}
for all $s\in[0,|\Omega|/2]$.
\end{theorem}
\begin{proof} We will borrow some ideas from \cite{Bramanti}.
To start with, we first notice that one can always reduce to consider smooth source data $f$, since in the less regular case we can obtain the estimate \eqref{concentrestim} through an approximation argument. According to \cite{dBVol}, using the change of variables
\[
z=\left(\frac{y}{2\sigma}\right)  ^{2\sigma},
\]
problems \eqref{eq.0.1} and \eqref{extensionsymm} become respectively

\begin{equation}
\left\{
\begin{array}
[c]{lll}%
-z^{\nu}\dfrac{\partial^{2}w}{\partial z^{2}}-\Delta_{x}w=0 &  & in\text{
}\mathcal{C}_{\Omega}\\
&  & \\
\dfrac{\partial w}{\partial \nu_{x}}=0 &  & on\text{ }\partial_{L}\mathcal{C}_{\Omega}\\
&  & \\
-\dfrac{\partial w}{\partial z}\left(  x,0\right)  =%
\beta_{\sigma}f\left(  x\right)  &  & in\text{ }\Omega,
\end{array}
\right.  \label{eq.0.1bis}%
\end{equation}

and

\begin{equation}
\left\{
\begin{array}
[c]{lll}%
-z^{\nu}\dfrac{\partial^{2}\xi}{\partial z^{2}}-\gamma\Delta_{x}\xi=0 &  & in\text{
}\mathcal{C}_{B}\\
&  & \\
\xi=0 &  & on\text{ }\partial_{L}\mathcal{C}_{B}\\
&  & \\
-\dfrac{\partial \xi}{\partial z}\left(  x,0\right)  =%
\beta_{\sigma}\left(f_{1}^{\#}\left(  x\right)+f_{2}^{\#}\left(  x\right)\right)  &  & in\text{ }B.
\end{array}
\right.  \label{extensionsymmbis}%
\end{equation}
where
$$\nu:=\left(  2\sigma-1\right)  /\sigma$$
and
\[
\beta_{\sigma}:=(2\sigma)^{2\sigma-1}\kappa_{\sigma}.
\]
Then, the problem reduces to prove the concentration comparison between the solutions $w(x,z)$ and $\xi(x,z)$ to \eqref{eq.0.1bis}-\eqref{extensionsymmbis} respectively.

Notice that by the weak formulation \eqref{weakformlevelbylevel} we have
\begin{equation}
\int_{\Omega}\nabla_{x}w(x,z)\cdot\nabla\theta(x)\,dx=z^{\nu}\int_{\Omega}\theta(x)\,w_{zz}(x,z)\,dx\label{weakformy}
\end{equation}
for all $\theta=\theta(x)\in H^{1}(\Omega)$ and a.e. $z>0$. Then,
let us fix $z>0,\,h>0$, $t\geq0$ and plug in \eqref{weakformy} the test function
\begin{equation*}
\varphi_{h,1}^{z}\left(  x\right)  =\left\{
\begin{array}
[c]{lll}%
1 &  & if\text{ \ }w_{1}(x,z) \geq t+h\\
&  & \\
\dfrac{ w_{1}(x,z) -t}{h}\, &  & if\text{
\ }t< w_{1}(x,z) <t+h\\
&  & \\
0 &  & if\text{ \ } w_{1}(x,z) \leq t.\text{ }%
\end{array}
\right.
\end{equation*}
Therefore, passing to the limit as $h\rightarrow0$
\begin{equation}
-z^{\nu}\int_{\left\{x:\,w_{1}(x,z)>t\right\}}\frac{\partial^{2}w}{\partial z^{2}}dx-\frac{d}{dt}\int_{\left\{x:\,w_{1}(x,z)>t\right\}}|\nabla_{x}w_{1}(x,z)|^{2}dx=0.\label{wekform}
\end{equation}
Using the relative isoperimetric inequality \eqref{relisop}, the coarea formula and the bounds \eqref{support} we have
\begin{align}
-\frac{d}{dt}\int_{\left\{x:\,w_{1}(x,z)>t\right\}}|\nabla_{x}w_{1}(x,z)|dx&=P(\left\{x:\,w_{1}(x,z)>t\right\};\Omega)\nonumber\\&\geq
Q^{-1}[\min\left\{\mu_{w_{1}}(t,z),|\Omega|-\mu_{w_{1}}(t,z)\right\}]^{1-1/N}\nonumber\\&=Q^{-1}[\mu_{w_{1}}(t,z)]^{1-1/N}\label{isop}.
\end{align}

Then inserting \eqref{isop} into \eqref{wekform} and using the Cauchy-Schwartz inequality,
\[
Q^{-2}[\mu_{w_{1}}(t,z)]^{2-2/N}\leq\ z^{\nu}\left(\int_{w_{1}(\cdot,z)>t}\frac{\partial^{2}w}{\partial z^{2}}dx\right)\left(-\frac{\partial\mu_{w_{1}}}{\partial
t}\right)
\]
hence a change of variables leads to
\[
-z^{\nu}\int_{\left\{x:\,w_{1}(x,z)>w_{1}^{\ast}(s,z)\right\}}\frac{\partial^{2}w}{\partial z^{2}}dx-Q^{-2}s^{2-2/N}\frac{\partial^{2}w_{1}^{\ast}}{\partial s^{2}}\leq0.
\]
Now, observe that on the set $\left\{x:w_{1}(x,z)>w_{1}^{\ast}(s,z)\right\}$ we have
\[
\frac{\partial^{2}w}{\partial z^{2}}=\frac{\partial^{2}w_{1}}{\partial z^{2}}+\lambda^{\prime\prime}(z)
\]
thus
\[
-z^{\nu}\int_{\left\{x:\,w_{1}(x,z)>w_{1}^{\ast}(s,z)\right\}}\left(\frac{\partial^{2}w_{1}}{\partial
z^{2}}+\lambda^{\prime\prime}(z)\right)dx-Q^{-2}s^{2-2/N}\frac{\partial^{2}w_{1}^{\ast}}{\partial s^{2}}\leq0,
\]
so the second order derivation formula \eqref{Ferone-Mercaldo} by Ferone-Mercaldo shows that
\begin{equation}
-z^{\nu}\int_{0}^{s}\left(\frac{\partial^{2}w_{1}^{\ast}}{\partial
z^{2}}+\lambda^{\prime\prime}(z)\right)d\tau-Q^{-2}s^{2-2/N}\frac{\partial^{2}w_{1}^{\ast}}{\partial s^{2}}\leq0,\label{firstestim}
\end{equation}
for a.e. $s\in (0,|\Omega|/2)$ and $z>0$.\\
Now we use in \eqref{weakformy} the test function
\begin{equation*}
\varphi_{h,2}^{z}\left(  x\right)  =\left\{
\begin{array}
[c]{lll}%
1 &  & if\text{ \ }w_{2}(x,z) \geq t+h\\
&  & \\
\dfrac{ w_{2}(x,z) -t}{h}\, &  & if\text{
\ }t< w_{2}(x,z) <t+h\\
&  & \\
0 &  & if\text{ \ } w_{2}(x,z) \leq t,\text{ }%
\end{array}
\right.
\end{equation*}
in order to obtain
\begin{equation}
-z^{\nu}\int_{\left\{x:\,w_{2}(x,z)>t\right\}}\frac{\partial^{2}w}{\partial z^{2}}dx+\frac{d}{dt}\int_{\left\{x:\,w_{2}(x,z)>t\right\}}|\nabla_{x}w_{2}(\cdot,z)|^{2}dx=0\label{wekform2}.
\end{equation}
The coarea formula and relative isoperimetric inequality  applied to $w_{2}(\cdot,z)$ give
\begin{align*}
-\frac{d}{dt}\int_{\left\{x:\,w_{2}(x,z)>t\right\}}|\nabla_{x}w_{2}(x,z)|dx\geq Q^{-1}[\mu_{w_{2}}(t,z)]^{1-1/N}
\end{align*}
then from \eqref{wekform2}
\[
\(-z^{\nu}\int_{\left\{x:\,w_{2}(x,z)>t\right\}}\frac{\partial^{2}w}{\partial z^{2}}dx\)\left(-\frac{\partial \mu_{w_{2}}}{\partial t}\right)\geq
Q^{-2}[\mu_{w_{2}}(t,z)]^{2-2/N}
\]
which yields
\begin{equation}
-Q^{-2}s^{2-2/N}\frac{\partial^{2}w_{2}^{\ast}}{\partial s^{2}}\leq -z^{\nu}\int_{\left\{x:\,w_{2}(x,z)>w_{2}^{\ast}(s,z)\right\}}\frac{\partial^{2}w}{\partial z^{2}}dx.\label{intermediate}
\end{equation}
Now, observe that on the set $\left\{x:w_{2}(x,z)>w_{2}^{\ast}(s,z)\right\}$ we have
\[
\frac{\partial^{2}w}{\partial z^{2}}=-\frac{\partial^{2}w_{2}}{\partial z^{2}}+\lambda^{\prime\prime}(z)
\]
hence by \eqref{intermediate}
\begin{equation*}
-Q^{-2}s^{2-2/N}\frac{\partial^{2}w_{2}^{\ast}}{\partial s^{2}}\leq z^{\nu}\int_{\left\{x:\,w_{2}(x,z)>w_{2}^{\ast}(s,z)\right\}}\left(\frac{\partial^{2}w_{2}}{\partial
z^{2}}-\lambda^{\prime\prime}(z)\right)dx
\end{equation*}
and by \eqref{Ferone-Mercaldo}
\begin{equation}
-Q^{-2}s^{2-2/N}\frac{\partial^{2}w_{2}^{\ast}}{\partial s^{2}}\leq z^{\nu}\int_{0}^{s}\left(\frac{\partial^{2}w_{2}^{\ast}}{\partial
z^{2}}-\lambda^{\prime\prime}(z)\right)d\tau.\label{secondestim}
\end{equation}
Then, adding \eqref{firstestim} to \eqref{secondestim} we have
\begin{equation}
-z^{\nu}\int_{0}^{s}\frac{\partial^{2}}{\partial z^{2}}(w_{1}^{\ast}+w_{2}^{\ast})d\tau-Q^{-2}s^{2-2/N}\frac{\partial^{2}}{\partial
s^{2}}(w_{1}^{\ast}+w_{2}^{\ast})\leq 0\label{thirdestimate}
\end{equation}
for a.e. $s\in (0,|\Omega|/2)$ and $z>0$.\\
Next, we set
\[
U(s,z)=\int_{0}^{s}(w_{1}^{\ast}(\tau,z)+w_{2}^{\ast}(\tau,z))d\tau.
\]
Now we observe that by the main result in \cite{Stinga-Torrea},
\[
-\lim_{z\rightarrow0^{+}}\frac{\partial w}{\partial z}(\cdot,y)=\beta_{\sigma}f\,\quad \,in\,\,L^{2}(\Omega)
\]
then
\[
-\frac{\partial w}{\partial z}(\cdot,z)=\beta_{\sigma}f+\mathcal{R}(\cdot,z)\,\quad \,in\,\,L^{2}(\Omega)\,, z\rightarrow0^{+}
\]
where the remaining term $\mathcal{R}=\mathcal{R}(x,z)$ is such that
\[
\lim_{z\rightarrow0^+}\mathcal{R}(\cdot,z)=0\quad \,in\,\,L^{2}(\Omega).
\]
Therefore using the first order derivation formula \eqref{Rakotoson} and the Hardy-Littlewood inequality \eqref{HardyLit} we have, for small $z>0$,
\begin{align}
-\frac{\partial U}{\partial z}(s,z)
&=-\int_{\left\{x:\,w_{1}(x,z)>w_{1}^{\ast}(s,z)\right\}}\frac{\partial w_{1}}{\partial
z}(x,z)dx-\int_{\left\{x:\,w_{2}(x,z)>w_{2}^{\ast}(s,z)\right\}}\frac{\partial w_{2}}{\partial z}(x,z)dx
\nonumber\\
&
=-\int_{\left\{x:\,w_{1}(x,z)>w_{1}^{\ast}(s,z)\right\}}\(\frac{\partial w}{\partial z}(x,z)-\lambda^{\prime}(z)\)dx-\int_{\left\{x:\,w_{2}(x,z)>w_{2}^{\ast}(s,z)\right\}}\(-\frac{\partial
w}{\partial z}(x,z)+\lambda^{\prime}(z)\)dx\nonumber\\
&=-\int_{\left\{x:\,w_{1}(x,z)>w_{1}^{\ast}(s,z)\right\}}\frac{\partial w}{\partial z}(x,z)\,dx+\int_{\left\{x:\,w_{2}(x,z)>w_{2}^{\ast}(s,z)\right\}}\frac{\partial
w}{\partial z}(x,z)dx\nonumber\\
&=\beta_{\sigma}\int_{\left\{x:\,w_{1}(x,z)>w_{1}^{\ast}(s,z)\right\}}f(x)\,dx-\beta_{\sigma}\int_{\left\{x:\,w_{2}(x,z)>w_{2}^{\ast}(s,z)\right\}}f(x)dx\nonumber\\
&+\int_{\left\{x:\,w_{1}(x,z)>w_{1}^{\ast}(s,z)\right\}}\mathcal{R}(x,z)\,dx-\int_{\left\{x:\,w_{2}(x,z)>w_{2}^{\ast}(s,z)\right\}}\mathcal{R}(x,z)\,dx\nonumber\\
&\leq \beta_{\sigma}\left(\int_{0}^{s}f_{1}^{\ast}(\tau)d\tau+\int_{0}^{s}f_{2}^{\ast}(\tau)d\tau\right)
+\int_{\left\{x:w_{1}(x,z)>w_{1}^{\ast}(s,z)\right\}}\mathcal{R}(x,z)\,dx-\int_{\left\{x: w_{2}(x,z)>w_{2}^{\ast}(s,z)\right\}}\mathcal{R}(x,z)\,dx
\label{mossinorak}
\end{align}
then passing to the limit as $z\rightarrow 0^{+}$ yields
\begin{align}
-\frac{\partial U}{\partial z}(s,0)
\leq\beta_{\sigma}\left(\int_{0}^{s}f_{1}^{\ast}(\tau)d\tau+\int_{0}^{s}f_{2}^{\ast}(\tau)d\tau\right).\label{mossinoraky}
\end{align}
Hence by inequalities \eqref{thirdestimate}, \eqref{mossinoraky} we find that $U$ satisfies
\begin{equation} \label{principsystem}
\left\{
\begin{array}
[c]{lll}%
-z^{\nu}\dfrac{\partial^{2}U}{\partial z^{2}}-Q^{-2}s^{2-2/N}\dfrac{\partial^{2}U}{\partial s^{2}}\leq0, \quad for\text{
}a.e.\,(s,z)\in (0,|\Omega|/2)\times(0,\infty)\\[15pt]
U(0,z)=\dfrac{\partial U}{\partial s}(|\Omega|/2,z)=0 \quad \forall z>0\\[15pt]
\dfrac{\partial U}{\partial{z}}(s,0)\geq -\beta_{\sigma}{\displaystyle\int_{0}^{s}(f_{1}^{\ast}(\tau)+f_{2}^{\ast}(\tau))d\tau},\quad\text{for a.e. } s\in
(0,|\Omega|/2).
\end{array}
\right. %
\end{equation}
Now, since the solution $\xi$ to \eqref{extensionsymm} is radially decreasing w.r. to $x$, all the inequalities used above become equalities, hence the function
\[
V(s,z)=\int_{0}^{s}\xi^{\ast}(\tau,z)d\tau.
\]
solves the problem
\begin{equation*}
\left\{
\begin{array}
[c]{lll}%
-z^{\nu}\dfrac{\partial^{2}V}{\partial z^{2}}-Q^{-2}s^{2-2/N}\dfrac{\partial^{2}V}{\partial s^{2}}=0, \quad for\text{
}a.e.\,(s,z)\in (0,|\Omega|/2)\times(0,\infty)\\[15pt]
V(0,z)=\dfrac{\partial V}{\partial s}(|\Omega|/2,z)=0 \quad \forall z>0\\[15pt]
\dfrac{\partial V}{\partial{z}}(s,0)= -\beta_{\sigma}{\displaystyle\int_{0}^{s}(f_{1}^{\ast}(\tau)+f_{2}^{\ast}(\tau))d\tau},\quad\text{for a.e. } s\in
(0,|\Omega|/2),
\end{array}
\right. %
\end{equation*}
thus the function
\[
\chi=U-V
\]
verifies
\begin{equation*}
\left\{
\begin{array}
[c]{lll}%
-z^{\nu}\dfrac{\partial^{2}\chi}{\partial z^{2}}-Q^{-2}s^{2-2/N}\dfrac{\partial^{2}\chi}{\partial s^{2}}\leq0, \quad for\text{
}a.e.\,(s,z)\in (0,|\Omega|/2)\times(0,\infty)\\[15pt]
\chi(0,z)=\dfrac{\partial \chi}{\partial s}(|\Omega|/2,z)=0 \quad \forall z>0\\[15pt]
\dfrac{\partial \chi}{\partial{y}}(s,0)\geq0,\quad\text{for a.e. } s\in
(0,|\Omega|/2).
\end{array}
\right. %
\end{equation*}
Then a classical maximum principle argument allows to conclude that
\[
\chi\leq0\quad\forall(s,z)\in[0,|\Omega|/2]\times[0,\infty)
\]
that is
\begin{equation*}
\int_{0}^{s}(w_{1}^{\ast}(\tau,z)+w_{2}^{\ast}(\tau,z))d\tau\leq
\int_{0}^{s}\xi^{\ast}(\tau,z)d\tau\quad\forall(s,z)\in[0,|\Omega|/2]\times[0,\infty).\label{mainestimate}
\end{equation*}
\end{proof}
\begin{remark}
\emph{
If the function $\lambda(z)$ in \eqref{functiont(y)} is constant, then Theorem \ref{comparisontheorem} can be actually strengthened. Indeed, in such a case the
second derivative of $\lambda(z)$ disappears in estimates \eqref{firstestim}, \eqref{secondestim}, so we have that the concentration function
\[
U_{i}(s,z)=\int_{0}^{s}w^{\ast}_{i}(\tau,y)d\tau\quad\forall i=1,2
\]
satisfies the system
\begin{equation*}
\left\{
\begin{array}
[c]{lll}%
-z^{\nu}\dfrac{\partial^{2}U_{i}}{\partial z^{2}}-Q^{-2}s^{2-2/N}\dfrac{\partial^{2}U_{i}}{\partial s^{2}}\leq0, \quad for\text{
}a.e.\,(s,z)\in (0,|\Omega|/2)\times(0,\infty)\\[15pt]
U_{i}(0,z)=\dfrac{\partial U_i}{\partial s}(|\Omega|/2,z)=0 \quad \forall z>0\\[15pt]
\dfrac{\partial U_{i}}{\partial{z}}(s,0)\geq -\beta_{\sigma}{\displaystyle\int_{0}^{s}f_{i}^{\ast}(\sigma)d\sigma},\quad\text{for a.e. } s\in
(0,|\Omega|/2).
\end{array}
\right. %
\end{equation*}
Then if for $i=1,2$ we call $v_{i}$ the solution to the problem
\begin{equation*} \label{symmetriz}
\left\{
\begin{array}
[c]{lll}%
\left(  -\gamma\Delta\right)^{\sigma}v_{i}=f_{i}^{\#}(x)   &  & in\text{ }%
B,\\ [10pt]
v_{i}=0  &  & on\text{ }\partial B
\end{array}
\right. %
\end{equation*}
and $\xi_{i}$ the Dirichlet extension of $v_{i}$, we have that
\[
V_{i}=\int_{0}^{s}\xi_{i}^{\ast}(\tau,z)d\tau
\]
solves
\begin{equation*}
\left\{
\begin{array}
[c]{lll}%
-z^{\nu}\dfrac{\partial^{2}V_{i}}{\partial z^{2}}-Q^{-2}s^{2-2/N}\dfrac{\partial^{2}V_{i}}{\partial s^{2}}=0, \quad for\text{
}a.e.\,(s,z)\in (0,|\Omega|/2)\times(0,\infty)\\[15pt]
V_{i}(0,z)=\dfrac{\partial V_i}{\partial s}(|\Omega|/2,z)=0 \quad \forall z>0\\[15pt]
\dfrac{\partial V_{i}}{\partial{y}}(s,0)= -\beta_{\sigma}{\displaystyle\int_{0}^{s}f_{i}^{\ast}(\tau)d\tau},\quad\text{for a.e. } s\in
(0,|\Omega|/2),
\end{array}
\right. %
\end{equation*}
therefore the function
\[
\chi_{i}=U_{i}-V_{i}
\]
is a solution to
\begin{equation*}
\left\{
\begin{array}
[c]{lll}%
-z^{\nu}\dfrac{\partial^{2}\chi_{i}}{\partial z^{2}}-Q^{-2}s^{2-2/N}\dfrac{\partial^{2}\chi_{i}}{\partial s^{2}}\leq0, \quad for\text{
}a.e.\,(s,z)\in (0,|\Omega|/2)\times(0,\infty)\\[15pt]
\chi_{i}(0,z)=\dfrac{\partial \chi_{i}}{\partial s}(|\Omega|/2,z)=0 \quad \forall z>0\\[15pt]
\dfrac{\partial \chi_{i}}{\partial{y}}(s,0)\geq0,\quad\text{for a.e. } s\in
(0,|\Omega|/2).
\end{array}
\right. %
\end{equation*}
By the maximum principle again we have
\[
\chi_{i}\leq0 \quad\forall(s,z)\in[0,|\Omega|/2]\times[0,\infty)
\]
that is
\[
\int_{0}^{s}w_{i}^{\ast}(\tau,z)d\tau\leq \int_{0}^{s}\xi_{i}^{\ast}(\tau,z)d\tau \quad\forall(s,z)\in[0,|\Omega|/2]\times[0,\infty).
\]
This can be interpreted as the mass concentration comparison version of the Maderna-Salsa result \cite{MadSalsa}, for the nonlocal operator $(-\Delta)^{\sigma}$.}
\end{remark}
Finally, a natural extension of Theorem \ref{comparisontheorem} is the following
\begin{corollary}\label{Corollary}
Assume that $g$ is a radially decreasing function on the ball $B$, such that
\[
f_{1}^{\#}+f_{2}^{\#}\prec g\,\,in\,\,B,
\]
and let $v$ be the solution to problem \eqref{symmetrizz} with $f_{1}^{\#}+f_{2}^{\#}$ replaced by $g$. If $\xi$ is the harmonic Neumann extension
of $v$ (namely the solution to \eqref{extensionsymm} with $f_{1}^{\#}+f_{2}^{\#}$ replaced by $g$), then the concentration inequality
\eqref{concentrestim} still holds.
\end{corollary}
\subsection{Consequences}
The following remarkable properties can be easily deduced from Theorem \ref{comparisontheorem}.\\[2pt]
\noindent {\sl 1. Oscillation estimate.} From the mass concentration comparison \eqref{concentrestim} we have
\begin{equation}
\int_{0}^{s}(u_{1}^{\ast}(\tau)+u_{2}^{\ast}(\tau))d\tau\leq
\int_{0}^{s}v^{\ast}(\tau)d\tau\quad\forall s\in[0,|\Omega|/2],\label{mainestimate2}
\end{equation}
where $u_{1}=(u-\mathsf{m}(u))^{+}$ and $u_{1}=(u-\mathsf{m}(u))^{-}$.
Then inequality \eqref{mainestimate2} can be rewritten as
\[
u_{1}^{\#}+u_{2}^{\#}\prec v,
\]
which implies, in particular, the meaningful oscillation estimate
\[
\|v\|_{L^{\infty}(B)}\geq\|u_{1}^{\#}+u_{2}^{\#}\|_{L^{\infty}(B)}=\sup_{\Omega}(u-\mathsf{m}(u))-\inf_{\Omega}(u-\mathsf{m}(u))=\sup_{\Omega} u-\inf_{\Omega} u.
\]
\noindent {\sl 2. $L^{p}$-estimates.} From \eqref{mainestimate2} we also have
\[
\int_{0}^{s}(u-\mathsf{m}(u))^{\ast}d\tau\leq\int_{0}^{s}(u_{1}^{\ast}(\tau)+u_{2}^{\ast}(\tau))d\tau\leq\int_{0}^{s}v^{\ast}(\tau)d\tau
\]
hence in particular
\[
\|u-\mathsf{m}(u)\|_{L^{p}(\Omega)}\leq\|v\|_{L^{p}(B)},
\]
for all $p\in [1,\infty)$. Then, making use of the fractional Dirichlet regularity estimates derived in \cite{dBVol}, we can obtain the whole sharp
$L^{p,r}$-scale
of regularity estimates for Neumann problems of the type \eqref{eq.0.1}, generalizing some of the assertions in \cite[Theorem 3.5]{StingaVolz} (see also \cite{Caffarelli-Stinga} for an important treatment of $C^{\alpha}$ regularity estimates up to the boundary). Therefore we can state the following result (for basic properties of Lorentz and Orlicz spaces see \emph{e.g} \cite{Bennett}):
\begin{theorem}\label{Thm:regularity}
Let $u\in H^{\sigma}(\Omega)$, $\sigma\in(0,1)$ be a solution to \eqref{eq.0.1} and $$f\in L^{p,r}(\Omega)$$ with
\[
p\geq\frac{2N}{N+2\sigma}, \quad r\geq1,
\]
where $L^{p,r}(\Omega)$ is the Lorentz space on $\Omega$ of exponents $p,\,r$. Suppose that $f$
verifies the compatibility condition \eqref{compatib}.
Then, for some positive constants $C$, the following assertions hold:
\begin{enumerate}
 \item if $p<N/2\sigma$ then $u\in L^{q,r}(\Omega)$ with
 \[
 q=\frac{Np}{N-2\sigma p}
 \]
 and
 \[
 \|u-\mathsf{m}(u)\|_{L^{q,r}(\Omega)}\leq C\|f\|_{L^{p,r}(\Omega)};
 \]
 \item if $p=N/2\sigma$ and $r=1$, then $u\in L^{\infty}(\Omega)$
and
 \[
 \|u-\mathsf{m}(u)\|_{L^{\infty}(\Omega)}\leq C\|f\|_{L^{N/2\sigma,1}(\Omega)};
 \]
\item if $p=N/2\sigma$ and $r\in(1,\infty]$, then $u\in L_{\Phi_{r}}(\Omega)$ and
\[
\|u-\mathsf{m}(u)\|_{L_{\Phi_{r}}(\Omega)}\leq C \|f\|_{L^{N/2\sigma,r}(\Omega)},
\]
where $L_{\Phi_{r}%
}(\Omega)$ is the Orlicz space generated by the $N$-function
\[
\Phi_{r}(t)=\exp(|t|^{r^{\prime}})-1
\]
being $r^{\prime}$ the conjugate exponent of $r$.
\end{enumerate}
\end{theorem}

\section{Extensions to operators with constant zero order coefficient}\label{zeroordersec}

If $c>0$ is a constant, we wish to generalize Theorem \ref{comparisontheorem} for fractional linear Neumann problems of the
type \eqref{eq.0}.
Of course, in this setting we will not require the compatibility condition \eqref{compatib}. According to what explained in Section \ref{Sec3}, the unique weak solution $u\in H^{\sigma}(\Omega)$ is the trace over $\Omega$ of the unique weak solution $w\in \mathsf{H}^{\sigma,c}(\mathcal{C}_{\Omega})$ to the extension problem \eqref{extenszero}.
In this case, we compare problem \eqref{eq.0} with the radial Dirichlet problem
\begin{equation} \label{symmetriz2}
\left\{
\begin{array}
[c]{lll}%
\left(  -\gamma\Delta\right)^{\sigma}v+cv=f_{1}^{\#}(x)+f_{2}^{\#}\left(  x\right)   &  & in\text{ }%
B,\\ [10pt]
v=0  &  & on\text{ }\partial B,
\end{array}
\right. %
\end{equation}
being
\[
f_{1}=(f-\mathsf{m}(f))^{+},\,f_{2}=(f-\mathsf{m}(f))^{-}.
\]
The extension problem associated to \eqref{symmetriz2} is given by
\begin{equation}
\left\{
\begin{array}
[c]{lll}%
\dfrac{\partial^{2}\xi}{\partial y^{2}}+\dfrac{1-2\sigma}{y}\dfrac{\partial\xi}{\partial y}+\gamma\Delta_{x}\xi=0 &  & in\text{
}\mathcal{C}_{B},\\[15pt]
\xi=0 &  & on\text{ }\partial_{L}\mathcal{C}_{B},\\[15pt]
\displaystyle{-\frac{1}{\kappa_{\sigma}}\lim_{y\rightarrow0^{+}}y^{1-2\sigma}\dfrac{\partial\xi}{\partial y}(x,y)}+c\,\xi(x,0) =f_{1}^{\#}(x)+f_{2}^{\#}(x)   &
&
in\text{ }B.
\end{array}
\right.\label{extensionsymme2}
\end{equation}
In this respect, we will prove the following result:
\begin{theorem}\label{comparisonc}
Let $c>0$,
assume that $f\in H^{\sigma}(\Omega)^{\prime}$ and let $u,\,w$ be the solutions to \eqref{eq.0} and the extension problem \eqref{extenszero} respectively. Let $v,\,\xi$ be the solutions to the symmetrized problem \eqref{symmetriz2} and the extension problem \eqref{extensionsymme2} respectively. Then inequality
\eqref{concentrestim} still holds.
\end{theorem}
\begin{proof}
It is crystal clear that the first estimate in \eqref{principsystem} with the first two boundary conditions still hold. As for the Neumann condition satisfied by $U$, we first observe that since (see \cite{Stinga-Torrea} again)
\[
\lim_{z\rightarrow0^{+}}\left[-\frac{\partial w}{\partial z}(\cdot,z)+\beta_{\sigma}cw(\cdot,z)\right]=\beta_{\sigma}f\,\quad \,in\,\,L^{2}(\Omega)
\]
we have
\[
-\frac{\partial w}{\partial z}(\cdot,z)+\beta_{\sigma}cw(\cdot,z)=\beta_{\sigma}f+\mathcal{R}(\cdot,z)\,\quad \,in\,\,L^{2}(\Omega)\,, z\rightarrow0^{+}
\]
for a certain remaining term $\mathcal{R}=\mathcal{R}(x,z)$ tending to zero in $L^{2}(\Omega)$, uniformly in $z$.
Then using the same notations of the proof of Theorem \ref{comparisontheorem} and arguing as in inequality \eqref{mossinorak} we have, for small $z>0$,
\begin{align}
-\frac{\partial U}{\partial z}(s,z)&=-\int_{\left\{x:\,w_{1}(x,z)>w_{1}^{\ast}(s,z)\right\}}\frac{\partial w}{\partial
z}(x,z)\,dx+\int_{\left\{x:\,w_{2}(x,z)>w_{2}^{\ast}(s,z)\right\}}\frac{\partial
w}{\partial z}(x,z)dx \nonumber\\
&=-\beta_{\sigma}c\int_{\left\{x:\,w_{1}(x,z)>w_{1}^{\ast}(s,z)\right\}}w(x,z)dx+\beta_{\sigma}c\int_{\left\{x:\,w_{2}(x,z)>w_{2}^{\ast}(s,z)\right\}}w(x,z)dx \nonumber\\&+\beta_{\sigma}\int_{\left\{x:\,w_{1}(x,z)>w_{1}^{\ast}(s,z)\right\}}f(x)dx
-\beta_{\sigma}\int_{\left\{x:\,w_{2}(x,z)>w_{2}^{\ast}(s,z)\right\}}f(x)dx\nonumber\\
&+\int_{\left\{x:\,w_{1}(x,z)>w_{1}^{\ast}(s,z)\right\}}\mathcal{R}(x,z)dx-\int_{\left\{x:\,w_{2}(x,z)>w_{2}^{\ast}(s,z)\right\}}\mathcal{R}(x,z)dx\nonumber\\
&=-\beta_{\sigma}c\int_{\left\{x:\,w_{1}(x,z)>w_{1}^{\ast}(s,z)\right\}}w_{1}(x,z)dx-\beta_{\sigma}c\int_{\left\{x:\,w_{2}(x,z)>w_{2}^{\ast}(s,z)\right\}}w_{2}(x,z)dx\nonumber
\\&+\beta_{\sigma}\int_{\left\{x:\,w_{1}(x,z)>w_{1}^{\ast}(s,z)\right\}}(f(x)-\mathsf{m}(f))dx
-\beta_{\sigma}\int_{\left\{x:\,w_{2}(x,z)>w_{2}^{\ast}(s,z)\right\}}(f(x)-\mathsf{m}(f))dx\nonumber\\
&+\int_{\left\{x:\,w_{1}(x,z)>w_{1}^{\ast}(s,z)\right\}}\mathcal{R}(x,z)dx-\int_{\left\{x:\,w_{2}(x,z)>w_{2}^{\ast}(s,z)\right\}}\mathcal{R}(x,z)dx\nonumber\\
&=-\beta_{\sigma}c\,U(s,z)+\beta_{\sigma}\int_{\left\{x:\,w_{1}(x,z)>w_{1}^{\ast}(s,z)\right\}}(f(x)-\mathsf{m}(f))dx
-\beta_{\sigma}\int_{\left\{x:\,w_{2}(x,z)>w_{2}^{\ast}(s,z)\right\}}(f(x)-\mathsf{m}(f))dx\label{zeroorder}\\
&+\int_{\left\{x:\,w_{1}(x,z)>w_{1}^{\ast}(s,z)\right\}}\mathcal{R}(x,z)dx-\int_{\left\{x:\,w_{2}(x,z)>w_{2}^{\ast}(s,z)\right\}}\mathcal{R}(x,z)dx\nonumber\
\end{align}
therefore using the Hardy-Littlewood inequality \eqref{HardyLit} and passing to the limit as $z\rightarrow0^{+}$ we find
\begin{equation}
-\frac{\partial U}{\partial z}(s,0)\leq-\beta_{\sigma}c\,U(s,0)+\beta_{\sigma}\left(\int_{0}^{s}f_{1}^{\ast}(\tau)d\tau+\int_{0}^{s}f_{2}^{\ast}(\tau)d\tau\right).\label{eq.3}
\end{equation}
Taking into account the symmetry of the solution $\xi$ to \label{extensionsymm2} with respect to $x$, the first inequality in \eqref{principsystem} and inequality \eqref{eq.3} become equalities, up to
replacing $U$ by the concentration $V$ of $\xi$, that is
\[
V(s,z)=\int_{0}^{s}\xi^{\ast}(\tau,z)d\tau.
\]
Then we finally obtain
\begin{equation*} \label{princsystem}
\left\{
\begin{array}
[c]{lll}%
-z^{\nu}\dfrac{\partial^{2}\chi}{\partial z^{2}}-Q^{-2}s^{2-2/N}\dfrac{\partial^{2}\chi}{\partial s^{2}}\leq0, \quad for\text{
}a.e.\,(s,z)\in (0,|\Omega|/2)\times(0,\infty)\\[15pt]
\chi(0,z)=\dfrac{\partial \chi}{\partial s}(|\Omega|/2,z)=0 \quad \forall z>0\\[15pt]
\dfrac{\partial \chi}{\partial{z}}(s,0)\geq \beta_{\sigma}c\chi(s,0),\quad\text{for a.e. } s\in
(0,|\Omega|/2),
\end{array}
\right. %
\end{equation*}
where have set as usual
\[
\chi=U-V.
\]
By Hopf's boundary maximum principle we easily obtain that $\chi\leq0$, which is the desired estimate.
\end{proof}

It is worth noting that an easy analogue of Corollary
\eqref{Corollary} occurs, which can be stated as follows:
\begin{corollary}\label{corollimplicitime}
Assume that $c>0$, $u$ solves the problem
\begin{equation*}
\left\{
\begin{array}
[c]{lll}%
\left(  -\Delta\right)^{\sigma}u+cu=f\left(  x\right)+h(x)   &  & in\text{ }%
\Omega,\\[10pt]
\dfrac{\partial u}{{\partial\nu}}=0 &  & on\text{ }\partial\Omega,
\end{array}
\right. %
\end{equation*}
and let $w$ its Neumann harmonic extension.
Let $g_{1},\,g_{2}$ be two radially decreasing functions on the ball $B$, such that
\[
f_{1}^{\#}+f_{2}^{\#}\prec g_{1}\,\,in\,\,B,
\]
where
\[
f_{1}=(f-\mathsf{m}(f))^{+},\,f_{2}=(f-\mathsf{m}(f))^{-}
\]
and
\[
(h^{+})^{\#}+(h^{-})^{\#}\prec g_{2}.
\]
Let $v$ be the solution to problem \eqref{symmetriz2} with $f_{1}^{\#}+f_{2}^{\#}$ replaced by $g_{1}+g_{2}$. If $\xi$ is the harmonic Neumann extension
of $v$ (namely the solution to \eqref{extensionsymme2} with $f_{1}^{\#}+f_{2}^{\#}$ replaced by $g_{1}+g_{2}$) then the concentration inequality
\eqref{concentrestim} still holds.
\end{corollary}
\begin{proof}
Just notice that from \eqref{zeroorder} we have
\begin{align}
-\frac{\partial U}{\partial z}(s,z)&=-\beta_{\sigma}c\,U(s,z)+\beta_{\sigma}\int_{\left\{x:w_{1}(x,z)>w_{1}^{\ast}(s,z)\right\}}(f(x)-\mathsf{m}(f))dx
-\beta_{\sigma}\int_{\left\{x:w_{2}(x,z)>w_{2}^{\ast}(s,z)\right\}}(f(x)-\mathsf{m}(f))dx\nonumber\\
&+\beta_{\sigma}\int_{\left\{x:w_{1}(x,z)>w_{1}^{\ast}(s,z)\right\}}h(x)dx-\int_{\left\{x:w_{2}(x,z)>w_{2}^{\ast}(s,z)\right\}}h(x)dx\nonumber\\
&+\int_{\left\{x:w_{1}(x,z)>w_{1}^{\ast}(s,z)\right\}}\mathcal{R}(x,z)dx-\int_{\left\{x:w_{2}(x,z)>w_{2}^{\ast}(s,z)\right\}}\mathcal{R}(x,z)dx\nonumber
\end{align}
thus Hardy-Littlewood inequality \eqref{HardyLit} yields, after passing to the limit as $z\rightarrow0^{+}$
\begin{align*}
-\frac{\partial U}{\partial z}(s,0)&\leq-\beta_{\sigma}c\,U(s,0)+\beta_{\sigma}\left(\int_{0}^{s}f_{1}^{\ast}(\tau)d\tau+\int_{0}^{s}f_{2}^{\ast}(\tau)d\tau\right)\\
&+\int_{0}^{s}(h^{+})^{\ast}(\tau)d\tau+\int_{0}^{s}(h^{-})^{\ast}(\tau)d\tau,
\end{align*}
thus we can argue as in the proof of Theorem \ref{comparisonc}
\end{proof}
\section{Symmetrization for linear fractional parabolic equations with Neumann boundary conditions}\label{parabolic}
Our goal now is to use the elliptic results shown in the previous sections in order to prove a symmetrization result for the following linear,
fractional parabolic Cauchy-Neumann problem
\begin{equation}
\left\{
\begin{array}
[c]{lll}%
u_{t}+(-\Delta)^{\sigma}u =f &  & in\text{
}\Omega\times(0,T)\\[15pt]
\dfrac{\partial u}{\partial\nu}=0 &  & on\text{ }\partial \Omega\times[0,T],\\[15pt]
u(x,0)=u_{0}(x)   & &
in\text{ }\Omega.
\end{array}
\right.  \label{parabolicpro}%
\end{equation}
In this framework, we will always assume that $u_{0}\in L^{2}(\Omega)$, $f\in L^{2}(\Omega\times(0,T))$.\\
It is easy to recast the issue of solving problem \eqref{parabolicpro} in an abstract setting. Indeed, the introduction of the linear operator
$\mathcal{A}_{N}:D(\mathcal{A})\subset H\rightarrow H$, where $H=L^{2}(\Omega)$, defined by
\[
\mathcal{A}_{N}u=(-\Delta_{N})^{\sigma}u
\]
with the fixed domain
\[
D(\mathcal{A}_{N})=H^{1}(\Omega)
\]
allows to reformulate the parabolic problem \eqref{parabolicpro} as the abstract Cauchy problem
\begin{equation}\label{abstract}
\left\{
\begin{array}
[c]{lll}%
u^{\prime}(t)+\mathcal{A}_{N}u(t)=f(t)   &  & on\,[0,T]\\
u(0)=u_{0},
\end{array}
\right.
\end{equation}
where as usual we have set $f(t)(x)=f(x,t)$. The concept of solution to problem \eqref{abstract} (or equivalently to \eqref{parabolicpro}), which suitably
adapts for the use of the elliptic symmetrization arguments proved in the previous sections, is that of \emph{mild} solution, namely a solution which is
obtained as the uniform limit of a time piecewise constant sequence of discrete approximated solutions, defined by means of an implicit time
discretization
scheme. In order to briefly introduce such definition, assume first to  divide the time interval $[0,T]$ in $n$ subintervals
$(t_{k-1},t_{k}]$, where $t_{k}=kh$ and $h=T/n$. Next we consider a time discretization  $\{f_k^{(h)}\}$ of $f$, such that the piecewise constant
interpolation of this  sequence provides a function $f^{(h)}(x,t)$ for which
$\|f-f^{(h)}\|_1\to 0$ as $h\to 0$. We construct then the function $u_{h}$ which is piecewise constant in each interval
$(t_{k-1},t_{k}]$,
by
\begin{equation}\label{approxsolut}
u_{h}(x,t)=
\left\{
\begin{array}
[c]{lll}%
u_{h,1}(x)  &  & if\,\,t\in[0,t_{1}]%
\\[6pt]
u_{h,2}(x)  &  & if\,\,t\in(t_{1},t_{2}]
\\
[6pt]
\cdots
\\
[6pt]
u_{h,n}(x)  &  & if\,\,t\in(t_{n-1},t_{n}]
\end{array}
\right. %
\end{equation}
where $u_{h,k}$ solves the elliptic equation
\begin{equation}
h\mathcal{A}_{N}(u_{h,k})+u_{h,k}=u_{h,k-1}+hf_k^{(h)}\label{discreteequat}
\end{equation}
with the initial value $u_{h,0}=u_{0}$. Then we wish to find that $u_{h}(t)$ converges as $h\rightarrow0$ to a certain function $u(t)$ uniformly in
$[0,T]$, where $u(t)$ is continuous at $t=0$ and $u(t)\rightarrow u_{0}$ as $t\rightarrow0$, namely we would like to prove that $u(t)$ is a \emph{mild
solution} to \eqref{parabolicpro}. The following Lemma gives a positive answer to such question.
\begin{lemma}\label{existence}
There exists a unique mild solution $u$ to problem \eqref{parabolicpro}.
\end{lemma}
\begin{proof}
Since we work in the Hilbert space $H=L^{2}(\Omega)$ for the linear operator\\ $\mathcal{A}_{N}:D(\mathcal{A}_{N})\subset H\rightarrow H$, it is sufficient to show
that $\mathcal{A}_{N}$ is maximal monotone. It is clear that $\mathcal{A_{N}}$ is monotone: indeed, for any $u\in H^{1}(\Omega)$ we have
\[
(\mathcal{A}_{N}u,u)_{L^2(\Omega)}=\|(-\Delta_{N})^{\sigma/2}\|^{2}_{L^{2}(\Omega)}\geq0.
\]
It is straightforward to check that $\mathcal{A}_{N}$ is self-adjoint.
Moreover, $\mathcal{A}$ is maximal monotone, \emph{i.e.} $R(I+\mathcal{A}_{N})=H$, since for any $f\in L^{2}(\Omega)$ there is a unique $u\in H^{1}(\Omega)$
solving (see \cite{StingaVolz}) the equation
\[
u+(-\Delta_{N})^{\sigma}u=f.
\]
Then \cite[Theorem 10.17]{VazPorous} ensures the existence of a unique mild solution $u\in C([0,T];L^{2}(\Omega)$ to the abstract Cauchy problem
\eqref{abstract}, obtained exactly as the uniform limit of the approximated solutions \eqref{approxsolut}.
\end{proof}
\begin{remark}\emph{
If there is no forcing term $f$ in \eqref{parabolicpro}, then the classical Hille-Yosida Theorem implies that for any $u_{0}\in L^{2}(\Omega)$ the mild
solution $u$ to \eqref{parabolicpro} is actually classical, and $u\in C^{1}([0,T];L^{2}(\Omega))\cap  C^{0}([0,T];H^{1}(\Omega))$ (see for instance
\cite{BREZIS}).}
\end{remark}
\begin{remark}\emph{
Keeping tracks of the papers \cite{afractpor,genfrac}, one could make use of the Stinga-Torrea extension method in order to give a proper meaning of
\emph{weak} energy solution to problem \eqref{parabolicpro}, which shows to coincide, when $f\equiv0$, to the unique mild solution obtained in Lemma
\ref{existence}. Nevertheless, in this context we decided to work only with mild solutions, which are enough for our purpose: the problem about the
equivalence of the two notions of solutions, along with several questions posed in a more general nonlinear setting, will be discussed in the forthcoming
paper \cite{BVprox}.}
\end{remark}
With these preliminaries at hand, we are in position to establish the following parabolic comparison result, related to problem \eqref{parabolicpro}.
\begin{theorem}
Assume that $u_{0}\in L^{2}(\Omega)$, $f\in L^{2}(\Omega\times (0,T))$ ($T>0$) and let $u$ be the mild solution to problem \eqref{parabolicpro}; set
\[
u_{1}(\cdot,t)=[u(\cdot,t)-\mathsf{m}(u(\cdot,t))]^{+},\,u_{2}(\cdot,t)=[u(\cdot,t)-\mathsf{m}(u(\cdot,t))]^{-}.
\]
Moreover, let $v$ be the mild solution  to the following Cauchy-Dirichlet problem, which is radially symmetric with respect to $x$:
\begin{equation}
\left\{
\begin{array}
[c]{lll}%
v_{t}+(-\gamma\Delta)^{\sigma}v =(f^{+})^{\#}+(f^{-})^{\#} &  & in\text{
}B\times(0,T)\\[15pt]
v=0 &  & on\text{ }\partial B\times[0,T],\\[15pt]
v(x,0)=u_{0,1}^{\#}(x)+u_{0,2}^{\#}(x)   & &
in\text{ }B,
\end{array}
\right.  \label{symmetrparabolic}%
\end{equation}
where $B$ is the ball centered at the origin with measure $|\Omega|/2$, $(f^{\pm})^{\#}(|x|,t)$ means symmetrization of $f^{\pm}(x,t)$ w.r. to $x$ for a.e.
time
$t>0$
and
\[
u_{0,1}=(u_{0}-\mathsf{m}(u_{0}))^{+},\,u_{0,2}=(u_{0}-\mathsf{m}(u_{0}))^{-}.
\]
Then,
for all $t>0$ we have
\begin{equation}
u_{1}^\#(|x|,t)+u_{2}^\#(|x|,t)\prec v(|x|,t).\label{comparisonparab}
\end{equation}
\end{theorem}
\begin{proof}
Let us consider the sequence of discrete approximated solutions \eqref{approxsolut}. Then applying the implicit time discretization scheme to the
symmetrized
problem \eqref{symmetrparabolic} produces the sequence of discrete solutions $v_{h}$ defined as
\[
v_{h}(x,t)=
\left\{
\begin{array}
[c]{lll}%
v_{h,1}(x)  &  & if\,t\in[0,t_{1}]%
\\
[6pt]
v_{h,2}(x)  &  & if\,t\in(t_{1},t_{2}]
\\
[6pt]
\cdots
\\
[6pt]
v_{h,n}(x)  &  & if\,t\in(t_{n-1},t_{n}],
\end{array}
\right. %
\]
where $v_{h,k}(x)$ solves the equation
\begin{equation}
h\gamma^{1/2}\mathcal{A}_{D}(v_{h,k})+v_{h,k}=v_{h,k-1}+hf_{k,1}^{(h)}+hf_{k,2}^{(h)}\label{discreteequsym}
\end{equation}
with the initial value $v_{h,0}=u_{0,1}^{\#}+u_{0,2}^{\#}$. The operator $\mathcal{A}_{D}:H_{0}^{1}(B)\subset L^{2}(\Omega)\rightarrow L^{2}(\Omega)$ is defined through $\mathcal{A}_{D}=(-\Delta_{D})^{\sigma}$ and we have posed $f_{k,1}^{(h)}=((f_{k}^{(h)})^{+})^{\#},\,f_{k,2}^{(h)}=((f_{k}^{(h)})^{-})^{\#}$.
Then
$v_{h}\rightarrow v$ uniformly, where $v$ is the unique mild solution to \eqref{symmetrparabolic}.\\
In order to compare $u$ with $v$, our aim is now to compare the solution $u_{h,k}$ to \eqref{discreteequat} with the solution $v_{h,k}$ to
\eqref{discreteequsym}. By induction we shall prove that
\begin{equation}
([u_{h,k}-\mathsf{m}(u_{h,k})]^{+})^{\#}+([u_{h,k}-\mathsf{m}(u_{h,k})]^{-})^{\#}\prec v_{h,k},\label{discretecomparison}
\end{equation}
for each $k=1,\ldots,n$. Indeed by Corollary \ref{corollimplicitime} we have
\[
([u_{h,1}-\mathsf{m}(u_{h,1})]^{+})^{\#}+([u_{h,1}-\mathsf{m}(u_{h,1})]^{-})^{\#}\prec v_{h,1}.
\]
If we suppose by induction that \eqref{discretecomparison} holds for $k-1$, we can use Corollary \ref{corollimplicitime} again and get
\eqref{discretecomparison}, which in turns implies
\begin{equation}
([u_{h}-\mathsf{m}(u_{h})]^{+})^{\#}+([u_{h}-\mathsf{m}(u_{h})]^{-})^{\#}\prec v_{h}.\label{discretecomparisoninterp}
\end{equation}
Then passing to the limit in \eqref{discretecomparisoninterp} as $h\rightarrow0^{+}$ the result follows.
\end{proof}
%%%%%%%%%%%%%%%%%%%%%%%%%%%%%%%%%%%%%%%%%%%%%%%%%%%%%

\section{Comments, extensions and open problems}

- It is worth noticing that all the results contained in this paper can be extended to a more general context, for example when the fractional Laplacian operator $\mathcal{L}=(-\Delta)^{\sigma}$ is replaced by a $\sigma$-power of a linear, second order elliptic operator in divergence form, namely defined by
\[
\mathcal{L}=L^{\sigma},
\]
where $\sigma\in(0,1)$, $L$ is defined by \eqref{ellipticop} and its coefficients $a_{ij}$ satisfy the ellipticity condition \eqref{elliptcond}. Indeed, \cite[Theorem 1.1]{Stinga-Torrea} allows to identify $\mathcal{L}$ with the Dirichlet-to-Neumann map defined by the extension problem
\begin{equation*}
\begin{cases}
\dfrac{\partial^{2}w}{\partial y^{2}}+\dfrac{1-2\sigma}{y}\dfrac{\partial\xi}{\partial y}+L_{x}w=0,&\hbox{in}~\mathcal{C}_{\Omega},\\
 \\
 \dfrac{\partial w}{\partial \nu_{x}}=0,&\hbox{on}~\partial_L\mathcal{C}_{\Omega},\\
 \\
 w(x,0)=u(x),&\hbox{on}~\Omega;
\end{cases}
\end{equation*}
moreover, the whole functional setting is carefully detailed in \cite{Caffarelli-Stinga}. Concerning the symmetrization techniques we just notice that, due to \eqref{elliptcond}, the equal sign in \eqref{wekform} is replaced by $\leq$. This allows to interpret Theorem \eqref{comparisontheorem}  as a full nonlocal version of the classical result of Maderna-Salsa \cite{MadSalsa}.

- It would be interesting to consider problems of the form \eqref{eq.0}, where $c$ is \emph{not} constant. Indeed, it is well known that in the local case this simple variation leads to further nontrivial issues, which can be solved by subtle, nonstandard modifications of the main results (see \emph{e.g.} \cite{AlTrMat}). Moreover, it would make sense to adapt our arguments when extra first order terms are added to the left-hand side of the same equation in \eqref{eq.0}. This will be an object of future studies. \\[0.5pt]

%%%%%%%%%%%%%%%%%%%%%%%%%%%%%%%%%%%%%%%%%%%%%%%%%%%%%%%%%%%%%%%%%%%
%%%%%%%%%%%%%%%%%%%%%%%%%%%%%%%%%%%%%%%%%%%%%%%%%%%%%%%%%%%%%%%%%%%

\noindent {\large\bf Acknowledgments}

\noindent B.V. is partially supported by the INDAM-GNAMPA project 2015 ``\textit{Proprietà qualitative di soluzioni di equazioni ellittiche e paraboliche}'' (ITALY).

%%%%%%%%%%%%%%%%%%%%%%%%%%%%%%%%%%%%%%%%%%%%%%%%%%%%%%%%%%%%%%%%%%%%%%%%%%%%%%%%%%%%%%%%%%%%%%%%%%%%%%%%%%%%%%%%%%%%%%%%%%%%%%%%%%%%%%%%%%%%%%%%%%%%%%%%%%%%%%%

%\bibliographystyle{siam}\small
%\bibliography{SymFractNeumbib}

%
2000 \textit{Mathematics Subject Classification.}
35B45,  % 	A priori estimates
35R11,   	% Fractional partial differential equations
35K20. %Initial-boundary value problems for second-order parabolic equations

%%%%%%%%%%%%%%%%%%%%%%%%%%%%%%%%%%%%%%%%%%%%%%%%%%%%%%%%%%%%%%%%%

%
\textit{Keywords and phrases.} Symmetrization, fractional Laplacian,
Neumann problems, nonlocal elliptic and parabolic equations.

\end{document}